\documentclass[intlimits]{amsart}
\pdfoutput=1

\usepackage[all,numberbysection,hyperref]{bi-discrete}
\usepackage{esint}
\renewcommand{\dashint}{\fint}

\usepackage[backend=biber,maxbibnames=5,maxalphanames=5,style=alphabetic,bibencoding=utf8,giveninits,url=false,isbn=false]{biblatex}
\AtBeginBibliography{\small}

\newcommand{\Jspw}{\mathcal{J}^s_{p,w}}
\newcommand{\Wspw}{{W^{s,p}_w(\RRn)}}
\newcommand{\Wspwz}{{W^{s,p}_{w,0}(\Omega)}}
\newcommand{\Lpw}{{L^p_w(\RRn)}}
\newcommand{\WspwO}{{W^{s,p}_w(\Omega)}}
\newcommand{\dwxy}{\frac{w(x)w(y)}{w(B_{x,y})}dx\,dy}

\renewcommand{\le}{\leqslant}
\renewcommand{\ge}{\geqslant}
\renewcommand{\leq}{\leqslant}
\renewcommand{\geq}{\geqslant}
\newcommand{\rom}[1]{\uppercase\expandafter{\romannumeral #1\relax}}
\newcommand{\ds}{\displaystyle}
\newcommand{\ep}{\epsilon}

\makeatletter
\@namedef{subjclassname@2020}{\textup{2020} Mathematics Subject Classification}
\makeatother

\begin{document}

\title{Nonlocal equations with degenerate weights}

\author{Linus Behn}
\address{Linus Behn, University Bielefeld, Universit\"atsstr. 25, 33615 Bielefeld, Germany}
\email{linus.behn@math.uni-bielefeld.de}

\author{Lars Diening}
\address{Lars Diening, University Bielefeld, Universit\"atsstr. 25, 33615 Bielefeld, Germany}
\email{lars.diening@uni-bielefeld.de}


\author{Jihoon Ok}
\address{Jihoon Ok, Department of Mathematics, Sogang University, Seoul 04107, Republic of Korea}
\email{jihoonok@sogang.ac.kr}

\author{Julian Rolfes}
\address{Julian Rolfes, University Bielefeld, Universit\"atsstr. 25, 33615 Bielefeld, Germany}
\email{julian.rolfes@uni-bielefeld.de}

\thanks{This work was partially funded by the Deutsche Forschungsgemeinschaft   (DFG, German Research Foundation) - SFB 1283/2 2021 - 317210226 (project A7) and IRTG 2235 (Project 282638148). J. Ok was supported by the National Research Foundation of Korea by the Korean Government (NRF-2022R1C1C1004523)}





\begin{abstract}
  We introduce fractional weighted Sobolev spaces with degenerate weights. For these spaces we provide embeddings and Poincaré inequalities. When the order of fractional differentiability goes to $0$ or $1$, we recover the weighted Lebesgue and Sobolev spaces with Muckenhoupt weights, respectively. Moreover, we prove interior Hölder continuity and Harnack inequalities for solutions to the corresponding weighted nonlocal integro-differential equations. This naturally extends a classical result by Fabes, Kenig, and Serapioni to the nonlinear, nonlocal setting.
\end{abstract}

\maketitle


\section{Introduction}
\label{sec:introduction}

Research on fractional nonlocal integro-differential equations has become one of the most prominent topics in the field of partial differential equations.  In particular, regularity theory for nonlocal equations associated with the fractional energy functional given by
\begin{align}\label{fractionalenergy}
 \int_{\RRn} \int_{\RRn} \bigg(\frac{|v(x)-v(y)|}{|x-y|^s}\bigg)^p k(x,y) \,dx \,dy
\end{align}
have been extensively studied over the last decade. In \eqref{fractionalenergy}, $s\in (0,1)$ denotes the the fractional differentiability parameter, $1<p < \infty$, and $k(x,y)$ is a nonnegative kernel on $\RR^n\times \RR^n$.
Note that, if $k(x,y)=|x-y|^{-n}$, then the Euler-Lagrange equation of \eqref{fractionalenergy} corresponds to the fractional $p$-Laplace equation $(-\Delta)^s_pu=0$ and when $p=2$ it corresponds to a linear equation.  The main purpose of this paper is to establish regularity theory for weak solutions to specific classes of degenerate or singular nonlocal equations that do not satisfy the standard ellipticity condition $k(x,y)\eqsim |x-y|^{-n}$.

The energy functional in \eqref{fractionalenergy} is the fractional counterpart to  the classical energy functional $\int_{\RR^n}  a (x) |\nabla v|^p \,dx$, where $a:\RR^n\to [0,\infty]$. The corresponding Euler-Lagrange equation is the following $p$-Laplace type equation:
\begin{align}\label{pLaplaceequation}
 \mathrm{div} (a(x)|Du|^{p-2}Du) =0.
\end{align}
Specifically, when $p=2$, this reduces to the linear equation
\begin{align}\label{linearequation}
 \mathrm{div} (A(x) Du) =0
\end{align}
with the matrix valued function $A(x)=a(x)I_{n}$, where $I_n$ is the $n\times n$ identity matrix. In general, for the linear equation \eqref{linearequation},
it is well-known that  if $A(x)$ satisfies the uniform ellipticity condition
\begin{align}
\lambda |\xi|^2 \le A(x)\xi \cdot  \xi \le \Lambda |\xi|^2
\end{align}
for some $0<\lambda\le \Lambda<\infty$, then the weak solution to \eqref{linearequation} is H\"older continuous, and satisfies Harnack's inequality by the De Giorgi-Nash-Moser theory (see, e.g., \cite{DeGiorgi57}). These results extend to the $p$-Laplace case, see \cite[Chapters 4.7 and 5.3]{LadyUralt68}.
An interesting and important question is the development of regularity theory when the coefficient $A(x)$ does not satisfy the uniform ellipticity condition. For the linear equation \eqref{linearequation}, Fabes, Kenig
and Serapioni \cite{FabKenSer82} proved Hölder continuity and Harnack's inequality for weak solutions to \eqref{linearequation} when $A(x)$ satisfies that
\begin{align*}
\lambda w(x)|\xi|^2 \le A(x)\xi \cdot  \xi \le \Lambda w(x)|\xi|^2,
\end{align*}
with a weight function $w$ belonging to the Muckenhoupt $A_2$ class. This result is sharp in the sense that there exists a weight $w$ with $w\in A_p$ for every $p>2$, such that the equation \eqref{linearequation} has non-H\"older continuous solution (see \cite[Section 2.3]{FabKenSer82}). We also refer to \cite{MS68,Trudinger71,Trudinger73,BellaSchaeffner21} for regularity results for degenerate linear equations, \cite{CruzMN13,Modica85,MRW15,BellaSchäffner23,CianchiSchaeffner2024} for the extension to  $p$-Laplace type problems, and \cite{BDGP22,BalciDieningByunLee23,CaoMP18,Stredulinsky-book} for  gradient estimates for degenerate equations.

For fractional nonlocal equations with energy functional \eqref{fractionalenergy}, the condition corresponding to the uniform ellipticity is that  $k(x,y)\eqsim \abs{x-y}^{-n}$. Under this condition, regularity results  have been explored for instance in \cite{CaffaSilve07,CCV11,BKassmann05,Kassmann07} for linear equations and \cite{DiCastroKuusiPalatucci14,DiCastroKuusiPalatucci16,Cozzi17} for fractional $p$-Laplace type equations.
In particular,  in \cite{DiCastroKuusiPalatucci16,DiCastroKuusiPalatucci14}, the authors obtain  sharp local H\"older estimates and Harnack inequality using nonlocal tail terms, by applying De Giorgi's approach to nonlocal problems. Since then there have been extensive research activities on regularity theory for fractional nonlocal equation.  For further regularity results for fractional nonlocal problems of the $p$-Laplace type, we refer to \cite{KuusiMingioneSire15,Schikorra16,BrascoLindgren17,BrascoLindgrenSchikorra18,IannMosconiSquassina16,DieningNowak23,ByunKim23,DieningKimLeeNowak24,GarainLindgren24,Boegeleinetal24} and for related results on fractional equations with more general growth conditions to \cite{ByunKimOk23,ChakerKimWeidner22,BehnDieningNowakScharle24}.

In this paper, we focus on fractional nonlocal problems with degenerate kernels $k(x,y)$. Our main model for the kernel in \eqref{fractionalenergy} is
\begin{align}\label{kernel_model}
k(x,y) = \frac{w(x)w(y)}{w(B_{x,y})} =  c_n\frac{w(x)w(y)}{\fint_{B_{x,y}}w(z)\,dz}\frac{1}{|x-y|^n},
 \end{align}
where  $w(x)$ is a weight, $w(B) = \int_B w(z)\,dz$ and $B_{x,y}\coloneqq B_{\frac 12 \abs{x-y}}(\frac{x+y}{2})$. We will make the assumption that $w$ is in the Muckenhoupt class $A_p$, which is natural in this context. Some background on Muckenhoupt weights is collected in Section \ref{ssec:basic-embedding}. Note that  the kernel $k$ in \eqref{kernel_model} does not satisfy the uniform ellipticity condition since $k(x,y)=0$ when $w(x)=0$ or $w(y)=0$.
Function spaces related to the energy \eqref{fractionalenergy} with $k(x,y)$ as given in \eqref{kernel_model} have been studied as interpolation spaces between Lebesgue and Sobolev spaces in the Euclidean space equipped with the measure $d\mu(x)=w(x)dx$. See \cite{GKS10,CDM19} for further details. However, fractional nonlocal equations associated with these energy functionals have not been systematically studied yet. To the best of our knowledge, even in the linear case, this is the first paper to investigate regularity theory for nonlocal equations with degenerate weights.

The paper consists of two major parts. In the first part, Section~\ref{sec:fract-weight-sobol}, we introduce and study fractional weighted Sobolev spaces with degenerate weights. We explore various properties, including the density of smooth functions and Sobolev-\Poincare~estimates, directly without relying on interpolation. We emphasize that the estimates obtained in this section are stable as the order of fractional differentiability $s$ approaches $1$.  To this end, we develop a novel Riesz-type potential estimate in Appendix~\ref{sec:appendix}, which is of independent interest.

The second part, Sections~\ref{sec:regularity}, focuses on regularity theory for associated degenerate nonlocal equations. We prove local boundedness (Theorem \ref{thm:bounded}), interior H\"older regularity (Theorem \ref{thm:Holder}), and Harnack's inequality (Theorem \ref{thm:Harnack}) for their weak solutions. This exactly corresponds to the results in \cite{FabKenSer82} for the classical linear equation \eqref{linearequation}.  We highlight that,  thanks to the $s$-stable Sobolev-\Poincare~inequality from Section~\ref{sec:fract-weight-sobol}, we can apply the nonlocal version of De Giorgi's approach, as developed in \cite{DiCastroKuusiPalatucci14,DiCastroKuusiPalatucci16}, and derive relevant regularity estimates that remain stable as $s$ approaches $1$. Finally, we would like to remark that  although we prove interior local H\"older continuity for homogenous problems, we expect that by the same approach we can also show global H\"older continuity for nonhomogeneous problems.

\section{Fractional weighted  Sobolev spaces}
\label{sec:fract-weight-sobol}

In this section we introduce nonlocal energies with degenerate weights and their corresponding function spaces. We show density of smooth, compactly supported functions and investigate compact embeddings. Moreover, we see that local weighted spaces are recovered if the order of fractional differentiability $s$ goes to $0$ resp. $1$. Finally, we prove a Sobolev-\Poincare~inequality which is stable as $s\nearrow 1$.

\subsection{Fractional, weighted energies}
\label{ssec:fract-weight-energ}

In this section we introduce our nonlocal energies.
We start with a bit of standard notation.

From now on let $n \geq 2$. By $L^1_{\loc}(\RRn)$ we denote the locally integrable functions.  A function $w \in L^1_{\loc}(\RRn)$ with $w >0$ almost everywhere is called a \emph{weight}. Then for a measurable set $U\subset \RRn$, we define the measure with respect to $w(x)\,dx$ by $w(U) \coloneqq \int_U w(x)\,dx$. By~$\abs{U}$ we denote the Lebesgue measure of~$U$.  By $L^p_w(\RRn)$ we denote the usual weighted Lebesgue space with norm $(\int_{\RRn} \abs{v}^p w\,dx)^{\frac 1p}$. Let $C^\infty_c(\RRn)$ and~$C^\infty_c(\Omega)$ denote the smooth, compactly supported functions on~$\RRn$ and~$\Omega \subset \RRn$, respectively. By $B_r(x)$ we denote a ball with center~$x$ and radius~$r$. For $x,y \in \RRn$ let $B_{x,y} \coloneqq B_{\frac 12 \abs{x-y}}(\frac{x+y}{2})$. Then $\overline{B_{x,y}}$ is the smallest (closed) ball containing $x$ and $y$. For a ball~$B$ and $\lambda >0$ we denote by $\lambda B$ the ball with the same center and $\lambda$ times the radius. For quantities $A$ and $B$, we write $A\lesssim B$ if $A\le c B$ for some universal constant $c>0$, and $A\eqsim B$ if $A\lesssim B$ and $B \lesssim A$.

From now on we will always assume that $s\in(0,1)$ and $1< p <\infty$. Let $w$ be a  weight on~$\RRn$. We define the \emph{fractional, weighted energy} $\Jspw\,:\, L^1_{\loc}(\RRn) \to [0, \infty]$ by
\begin{subequations}\label{eq:Jspw}
  \begin{align}
    \label{eq:Jspw-a}
  \mathcal{J}_{p,w}^s(v)  \coloneqq  c_s \int_{\RRn} \int_{\RRn} \bigg(\frac{|v(x)-v(y)|}{|x-y|^s}\bigg)^p k(x,y) \,dx \,dy,
\end{align}
where $k\,:\, \RRn \times \RRn \to [0,\infty)$ satisfies
\begin{align}
  \label{eq:cond-k}
  \lambda \frac{w(x)w(y)}{w(B_{x,y})} \leq
  k(x,y) \leq \Lambda \frac{w(x)w(y)}{w(B_{x,y})}
\end{align}
\end{subequations}
for some $0<\lambda \leq \Lambda < \infty$ and $c_s \coloneqq s(1-s)$. Note that strictly speaking the energy depends on $k$ and only indirectly on $w$. This will not play any role for the function spaces investigated in this section. However, it will become important for the minimizers studied in Section \ref{sec:regularity}.

For $w=1$ and $\lambda = \Lambda$ we therefore obtain the energy corresponding to the usual fractional Sobolev space~$W^{s,p}(\RRn)$. The constant~$c_s=s(1-s)$ will be used to keep track of the stability for~$s \nearrow 1$ and sometimes also for $s \searrow 0$.

\subsection{Fractional weighted Sobolev spaces}
\label{ssec:fract-weight-sobol-1}

In this section we introduce the fractional, weighted Sobolev spaces.  The energy $\Jspw$ from \eqref{eq:Jspw} is strictly convex and can be used to define the fractional, weighted Sobolev space
\begin{align*}
  \Wspw &\coloneqq \set{ v \in L^1_{\loc}(\RRn)\,:\, \Jspw(v) < \infty}.
\end{align*}
Note that
\begin{align*}
  \abs{v}_{\Wspw} \coloneqq  \big(\Jspw(v)\big)^{\frac 1p}
\end{align*}
defines a semi-norm on~$\Wspw$ and $\abs{v}_{\Wspw}=0$ if and only if~$v$ is constant. To obtain a norm, one could either take the  quotient space or add a norm like $\norm{v}_{L^p_w(B_1(0))}$.

Furthermore, for an open and bounded set $\Omega \subset \RRn$ we define
\begin{align*}
  \Wspwz &\coloneqq \set{ v \in \Wspw\,:\, v|_{\Omega^c} = 0}.
\end{align*}
Then $(\Wspwz, |\cdot|_{\Wspw})$ is a normed space. Note that in this case we still integrate over $\RRn \times \RRn$. In some situations we consider functions that are only defined on some open set $\Omega$. For this we set for a measurable set $M\subset \RRn \times \RRn$ and  measurable function $v$
\begin{align}\label{eq:JAB}
  \Jspw (v\,|\, M)\coloneqq c_s \int_{\RRn} \int_{\RRn} \indicatorset{(x,y)\in M}\bigg(\frac{|v(x)-v(y)|}{|x-y|^s}\bigg)^p k(x,y) \,dx \,dy .
\end{align}
For example, we have $\Jspw (v)= \Jspw(v\,|\, \RRn \times \RRn)$. Suppose in the following that~$\Omega \subset \RRn$ is open and bounded. At this step, we could include more general $\Omega$, but for simplicity we restrict ourselves to the case that we need later in the application. We then define
\begin{align}
  \Jspw (v\,|\, \Omega) \coloneqq \Jspw (v\, | \, \Omega \times \Omega),
\end{align}
and $\WspwO \coloneqq \set{v\in L^p(\Omega): \Jspw (v \,|\, \Omega)<\infty}$. We now investigate the properties of the space $\Wspwz$.

\begin{lemma}
  \label{lem:Banachspace}
  Let $1<p<\infty$, $s \in (0,1)$ and~$\Omega \subset \RRn$ be open and bounded. Then $(\Wspwz,\abs{\cdot}_{\Wspw})$ is a uniformly convex Banach space.
\end{lemma}
\begin{proof}
  We can map every $u \in \Wspwz$ to $\bar{u}(x,y) \coloneqq u(x) - u(y)$. This defines a (non-bijective) isometry~$T$ from $\Wspwz$ to the weighted space $L^p_{\abs{x-y}^{-s} k(x,y)}(\RRn \times \RRn) \eqqcolon Z$. Note that~$Z$ is a uniformly convex Banach space. Let $u_m \in \Wspwz$ be a Cauchy sequence. Then $\bar{u}_m \coloneqq T u_m$ is Cauchy in~$Z$ and there exists a limit~$\tilde{u} \in Z$. Moreover, by passing to a subsequence, we have $\bar{u}_m \to \tilde{u}$ almost everywhere. Since $\tilde{u}_m(x,\zeta)=\tilde{u}_m(x,z)$ for almost all $\zeta,z \in \Omega^c$, we have $\tilde{u}(x,\zeta)=\tilde{u}(x,z)$ for almost all $\zeta,z \in \Omega^c$. This allows to define $u(x) \coloneqq \tilde{u}(x,z)$ for (arbitrary)~$z \in \Omega^c$. Moreover, $\bar{u} = \tilde{u}$, since
  \begin{align*}
    \bar{u}(x,y) &= u(x)-u(y) = \bar{u}(x,z) - \bar{u}(y,z) = \lim_{m \to \infty} \big( \bar{u}_m(x,z) - \bar{u}_m(y,z)\big)
    \\
    &= \lim_{m \to \infty}(u_m(x) -u_m(y))
    = \lim_{m \to \infty} \bar{u}_m(x,y) = \tilde{u}(x,y).
  \end{align*}
  Since $\bar{u}_m \to \bar{u}$ and~$T$ is an isometry, we conclude that $u_m \to u$ in $\Wspwz$. Thus, $\Wspwz$ is a Banach space. Since~$Z$ is uniformly convex, so is~$\Wspwz$.
\end{proof}
\begin{lemma}
  \label{lem:embedding-simple}
  Let $1<p<\infty$, $s \in (0,1)$ and~$\Omega \subset \RRn$ be open and bounded. Then  $\Wspwz \embedding L^p_w(\RRn)$.
\end{lemma}
\begin{proof}
  Let $B$ be a ball with radius~$r$ containing~$\Omega$.  For $v \in \Wspwz$ we have
  \begin{align*}
    \Jspw(v) &\gtrsim c_s \int_B \int_{3B\setminus 2B}
    \bigg(\frac{|v(x)-v(y)|}{|x-y|^s}\bigg)^p \frac{w(x)w(y)}{w(B_{x,y})}\,dx \,dy
    \\
    &\gtrsim c_s r^{-sp} \int_B \int_{3B \setminus 2B}
    |v(y)|^p \frac{w(x)w(y)}{w(3B)}\,dx\,dy
    \\
    &\gtrsim c_s r^{-sp} \frac{w(3B\setminus 2B)}{w(3B)} \int_B
    |v(y)|^p w(y)\,dy.
  \end{align*}
  This proves the claim.
\end{proof}
Note that the embedding constant in Lemma~\ref{lem:embedding-simple} deteriorates for~$s \nearrow 1$. This is in contrast to our \Poincare{} type inequality of Theorem~\ref{thm:poincare}, which however requires  more assumptions on our weight.

\subsection{Basic embedding}
\label{ssec:basic-embedding}

In this section we show the $C^\infty_c(\Omega) \subset \Wspwz$. For this inclusion it would be sufficient to assume that our weight is doubling. However, in view of our application we already now make the stronger assumption that $w$ is of \emph{Muckenhoupt class} $A_p$. We therefore recall now the class of Muckenhoupt weights, which is the assumption on our weights that we will use for the rest of this paper. For this let $p\in (1,\infty)$. A weight~$w$ on~$\RRn$ is said to be an $A_p$-weight (in short $w \in A_p$) if and only if
\begin{align}\label{eq:Ap}
  [w]_{A_p} \coloneqq \sup_{B \subset \RRn} \Bigg(\dashint_B w\,dx \bigg(\dashint_B w^{\frac{1}{1-p}}\,dx\bigg)^{p-1} \Bigg) < \infty.
\end{align}
The class $A_p$ contains exactly the weights that make the maximal operator continuous on $L^p_w(\RRn)$. A standard example is given by $\abs{x}^\gamma \in A_p$ for all $-n<\gamma <(p-1)n$.
Note that every $A_p$ weight $w$ is \emph{doubling}, i.e., that for some $c_w >0$ we have
\begin{align}
  \label{eq:doubling}
  w(2B) \leq c_w w(B) \qquad \text{for all balls~$B$},
\end{align}
see Remark \ref{rem:doubling}. Let us recall a few standard properties of Muckenhoupt weights, see~\cite{Gra14}. It is often useful, to define the dual weight $\sigma \coloneqq w^{\frac{1}{1-p}}$. Then $\sigma \in A_{p'}$, where $\frac 1p + \frac 1{p'}=1$ and $ [w]_{A_p} = [\sigma]_{A_{p'}}^{p-1}$. Moreover, $w$ and~$\sigma$ are doubling with doubling constants $c_w = 2^{np} [w]_{A_p}$ and $c_\sigma = 2^{np'} [w]_{A_p}^{1/(p-1)}$, respectively. For every ball $B$, \eqref{eq:Ap} and Jensen's inequality imply
\begin{align*}
  1 \leq \dashint_B w\,dx \bigg( \dashint_B \sigma\bigg)^{p-1} \leq [w]_{A_p}.
\end{align*}
Hence, for all $x,y \in \RRn$
\begin{align}
  \label{eq:w-vs-sigma-B}
  |B_{x,y}| \le w(B_{x,y})^{\frac 1p} \sigma(B_{x,y})^{\frac 1{p'}} \le [w]_{A_p} \abs{B_{x,y}} \eqsim [w]_{A_p}  \abs{x-y}^n,
\end{align}
where the implicit constant only depends on~$n$.
\begin{lemma}
  \label{lem:doubling-annulus}
  Let $p\in (1,\infty)$ and $w\in A_p$. Then
  for all $x,y\in \RRn$ and $r>0$  with $r \leq \abs{x-y} \leq 2r$ we have
  \begin{align}
    w(B_{x,y}) \leq 2^{np}[w]_{A_p} w(B_r(x)), \quad\text{and}\quad w(B_{2r}(x)) \leq 5^{np}[w]_{A_p} w(B_{x,y}).
  \end{align}
\end{lemma}
\begin{proof}
  For every $\lambda >1$ and all balls $B$ we have $w(\lambda B)\leq \lambda^{np}[w]_{A_p}w(B)$, see \cite[Proposition 9.1.5]{Gramodern14}. With this we estimate
  \begin{alignat*}{3}
    w(B_{x,y}) &\leq w(B_{2r}(x)) &&\leq 2^{np}[w]_{A_p} w(B_r(x)),
    \\
    \text{and}\qquad w(B_{2r}(x)) &\leq w( 5B_{x,y}) &&\leq 5^{np}[w]_{A_p} w(B_{x,y}).
  \end{alignat*}
  This proves the claim.
\end{proof}
\begin{lemma}\label{lem:intweight}
  Let $p\in (1,\infty)$ and $w\in A_p$. Then for any $\alpha \in\RR$ with $\abs{\alpha} \leq \alpha_0$ and every ball $B_r(x)\subset \RRn$, we have
  \begin{alignat}{2}
    \label{eq:intweight1}
    \int_{B_r(x)} |x-y|^{\alpha} \frac{w(y)}{w(B_{x,y})}\,dy \eqsim \frac{r^\alpha}{\alpha} &\qquad&\text{for $\alpha>0$},
    \\
    \label{eq:intweight2}
    \int_{\RRn \setminus B_r(x)} |x-y|^\alpha \frac{w(y)}{w(B_{x,y})}\,dy \eqsim \frac{r^\alpha}{\abs{\alpha}} &\qquad&\text{for $\alpha<0$}.
  \end{alignat}
  The hidden constants depend continuously on $n$, $p$, $\alpha_0$, and linearly on $[w]_{A_p}$.
\end{lemma}
\begin{proof}
  Let $B = B_r(x)$. Splitting $B_r(x)$ into annuli, we estimate with Lemma~\ref{lem:doubling-annulus} for~$\alpha>0$.
  \begin{align*}
    \int_{B_r(x)} |x-y|^{\alpha} \frac{w(y)}{w(B_{x,y})}\,dy
    &\lesssim [w]_{A_p}\sum_{j \geq 0} (2^{-j\alpha}r^\alpha) \frac{w(2^{-j}B \setminus 2^{-j-1} B)}{w(2^{-j} B)} \leq \frac{[w]_{A_p}r^\alpha}{1-2^{-\alpha}}.
  \end{align*}
  For the reverse estimate we have
  \begin{align}\label{eq:proof_integrate}
    \int_{B_r(x)} |x-y|^{\alpha} \frac{w(y)}{w(B_{x,y})}\,dy
    &\geq \sum_{j \geq 0} 2^{-(j+1)\alpha}r^\alpha\!\!\! \int_{2^{-j}B\setminus 2^{-j-1}B}\frac{w(y)}{w(B_{x,y})}\,dy
  \end{align}
  Let $\widetilde{B}_j$ be any ball of maximal radius contained in $2^{-j}B\setminus 2^{-j-1}B$. Then we have $B_{x,y}\subset 7\widetilde{B}_j$. Thus
  \begin{align*}
    w(B_{x,y})\leq w(7\widetilde{B}_j)\leq 7^{np}[w]_{A_p}w(\widetilde{B}_j).
  \end{align*}
  Combining this with \eqref{eq:proof_integrate}, we arrive at
  \begin{align*}
    [w]_{A_p}\int_{B_r(x)} |x-y|^{\alpha} \frac{w(y)}{w(B_{x,y})}\,dy &\geq 7^{-np} \sum_{j \geq 0} 2^{-(j+1)\alpha}r^\alpha \frac{w(2^{-j}B\setminus 2^{-j-1}B)}{w(\widetilde{B}_j)}
    \\
    & \geq  7^{-np} \frac{1}{2^{\alpha}}\frac{r^{\alpha}}{1-2^{-\alpha}}
  \end{align*}
  Note that $\abs{1- 2^{-t}} \eqsim t$ for all $t \in [0,\alpha_0]$ with constant depending on~$\alpha_0$. Hence, $\frac{1}{1-2^{-\abs{\alpha}}} \eqsim \frac{1}{\abs{\alpha}}$.
  This proves the claim for $\alpha >0$. The proof for $\alpha <0$ follows by the same method.
\end{proof}

\begin{proposition}
  Let $p\in (1,\infty)$, $w\in A_p$ and~$\Omega \subset \RRn$ be open and bounded. Then
  $C^\infty_c(\RRn) \subset W^{s,p}_w(\RRn)$
  and $C^\infty_c(\Omega) \embedding \Wspwz$.
\end{proposition}
\begin{proof}
  Let $v \in C^\infty_c(\RRn)$. Choose a ball~$B$ with radius~$r$ and $\support v \subset B$. Then $\abs{v(x)-v(y)} \leq \abs{x-y}\, \norm{\nabla v}_\infty$ implies
  \begin{align*}
    \Jspw(v)
    &= c_s
    \int_{\RRn} \int_{\RRn} \bigg(\frac{|v(x)-v(y)|}{|x-y|^s}\bigg)^p\frac{w(x)w(y)}{w(B_{x,y})} \,dx  \,dy
    \\
    &\leq c_s \norm{\nabla v}_\infty^p \int_{2B} \int_{2B}  \abs{x-y}^{p-sp}\frac{w(x)}{w(B_{x,y})}\,dx\, w(y)\,dy
    \\
    &\quad + 2  \norm{v}_\infty^p \int_{2B} \int_{(2B)^c} \abs{x-y}^{-sp}\frac{w(x)}{w(B_{x,y})}\,dx \,w(y)\,dy.
  \end{align*}
  By Lemma~\ref{lem:intweight} we obtain
  \begin{align*}
    \Jspw(v)
    &\lesssim \norm{\nabla v}_\infty^p r^{p-sp} w(2B)
    + \norm{v}_\infty^p r^{-sp} w(2B) < \infty.
  \end{align*}
  This proves~$C^\infty_c(\RRn) \subset \Wspw$. Hence, we also have $C^\infty_c(\Omega) \subset \Wspwz$. Suppose now that $v_m \to v$ in $C^\infty_c(\RRn)$. Then by the calculations above
  \begin{align*}
    \Jspw(v_m-v)
    &\lesssim \norm{\nabla (v_m-v)}_\infty^p r^{p-sp} w(2B)
    + \norm{v_m-v}_\infty^p r^{-sp} w(2B) \xrightarrow{m} 0,
  \end{align*}
  which proves $C^\infty_c(\Omega) \embedding \Wspwz$.
\end{proof}

\subsection{\texorpdfstring{The limit as $s \to 0$ and $s \to 1$}{The limit as s -> 0 and s -> 1}}
\label{ssec:limit-s}

In this section we investigate what happens if the order of fractional differentiability~$s$ converges to $0$ or $1$. We see that we recover the energies of the weighted Lebesgue space~$L^p_w$ and of the weighted Sobolev space~$W^{1,p}_w$, respectively. This justifies the notation~$\Wspw$ and that our model can be seen as the natural intermediate space.
For the case of fractional Sobolev spaces, i.e., for $w=1$, corresponding results were shown in \cite{BourgainBrezisMiro2001,MazyaSha02,FoghemKassmannVoigt2020}.
\begin{theorem}
  \label{thm:limit-s}
  Let $p\in (1,\infty)$, $w\in A_p$, and $v \in C^\infty_c(\RRn)$. Then
  \begin{alignat*}{2}
    \liminf_{s \searrow 0} \Jspw(v) &\eqsim  \limsup_{s \searrow 0} \Jspw(v) &&\eqsim \int_{\RRn} \abs{v(z)}^p w(x)\,dx,
    \\
    \liminf_{s \nearrow 1} \Jspw(v) &\eqsim     \limsup_{s \nearrow 1} \Jspw(v) & &\eqsim \int_{\RRn} \abs{\nabla v(z)}^p w(x)\,dx.
  \end{alignat*}
  The implicit constants are only dependent on $n$, $p$, and $[w]_{A_p}$.
\end{theorem}
\begin{proof}
  Recall that the normalizing constant of $\Jspw$ is given by $c_s=s(1-s)$. Fix $v \in C^\infty_c(\RRn)$. Choose~$r \geq 1$ such that $\support v \subset B_r(0)\coloneqq B$.
  We begin with $s \nearrow 1$. We have $ \mathcal{J}^{s}_{p,w}(v) \eqsim \mathrm{I}_1 + \mathrm{I}_2 + \mathrm{I}_3$, where
  \begin{align*}
    \mathrm{I}_1 &\coloneqq
                   c_s \int_{\RRn} \int_{B_1(x)} \bigg(\frac{|v(x)-v(y)|}{|x-y|^s}\bigg)^p\frac{w(x)w(y)}{w(B_{x,y})}\,dy\,dx,
    \\
    \mathrm{I}_2 &\coloneqq
                  c_s \int_{2B} \int_{\RRn \setminus B_1(x)} \bigg(\frac{|v(x)-v(y)|}{|x-y|^s}\bigg)^p\frac{w(x)w(y)}{w(B_{x,y})}\, dy\,dx,
    \\
    \mathrm{I}_3 &\coloneqq
                  c_s \int_{\RRn \setminus 2B} \int_{B \setminus B_1(x)} \bigg(\frac{|v(y)|}{|x-y|^s}\bigg)^p\frac{w(x)w(y)}{w(B_{x,y})}\,dy\,dx.
  \end{align*}
  We first prove that $\mathrm{I}_2, \mathrm{I}_3 \to 0$ for $s \nearrow 1$. We estimate with Lemma~\ref{lem:intweight}
  \begin{align*}
    \mathrm{I}_2
    &\le
      (1-s) (2\norm{v}_\infty)^p \int_{2B} w(x) \bigg(\int_{\RRn \setminus B_1(x)} \abs{x-y}^{-sp} \frac{w(y)}{w(B_{x,y})}\,dy\bigg)\,dx
      \\
    &\eqsim
      (1-s)\frac{1}{sp} \norm{v}_\infty^p w(2B) \xrightarrow{s \to 1} 0
  \end{align*}
  and
  \begin{align*}
    \mathrm{I}_3
    &\leq
     (1-s) \|v\|_\infty^p  \int_{B} \int_{\RRn \setminus 2B} |x-y|^{-sp}\frac{w(x)}{w(B_{x,y})} \,dx\, w(y)\,dy
      \\
    &\eqsim
      (1-s)\norm{v}_\infty^p \frac{r^{-sp}}{sp} w(B)  \xrightarrow{s \to 1} 0.
  \end{align*}
  It remains to calculate the limit of~$\mathrm{I}_1$. We obtain
  \begin{align*}
    \mathrm{I}_1 &= c_s \int_{B_{r+1}(0)} \int_{B_1(0)} \bigg( \frac{\abs{v(x+h) - v(x)}}{\abs{h}^s} \bigg)^p \frac{w(x) w(x+h)}{w(B_{x,x+h})} \,dh \,dx.
  \end{align*}
  Let us put
  \begin{align*}
    \mathrm{I}_{1,1} &\coloneqq
    c_s \int_{B_{r+1}(0)} \int_{B_1(0)} \bigg( \frac{\abs{\skp{\nabla v(x)}{h}}}{\abs{h}^s} \bigg)^p \frac{w(x) w(x+h)}{w(B_{x,x+h})} \,dh \,dx,
  \end{align*}
  and define $\mathrm{I}_{1,2}\coloneqq \mathrm{I}_1 - \mathrm{I}_{1,1}$. By Talyor's theorem
  \begin{align*}
    \bigabs{v(x+h) - v(x) - \skp{\nabla v(x)}{h}} &\leq \norm{\nabla^2 v}_\infty \abs{h}^2.
  \end{align*}
  Using this and the elementary inequality $\abs{a^p-b^p}\leq p \abs{a-b}(a^{p-1}+b^{p-1})$ for $a,b\geq 0$ and $p> 1$ we estimate a term appearing in the integrand of $\mathrm{I}_{1,2}$:
  \begin{align*}
    &\bigabs{\abs{v(x+h)-v(x)}^p-\abs{\skp{\nabla v(x)}{h}}^p}
    \\
    &\qquad\lesssim \bigabs{v(x+h) - v(x) - \skp{\nabla v(x)}{h}} \left(\abs{v(x+h)-v(x)}+\abs{\skp{\nabla v(x)}{h}}\right)^{p-1}
    \\
    &\qquad\lesssim \norm{\nabla^2 v}_{L^\infty}\abs{h}^2 \norm{\nabla v}_{L^\infty}^{p-1}\abs{h}^{p-1}.
  \end{align*}
  Thus
  \begin{align*}
    \mathrm{I}_{1,2} &\lesssim
    c_s \int_{B_{r+1}(0)} \int_{B_1(0)} \frac{\norm{\nabla^2 v}_\infty \norm{\nabla v}_\infty^{p-1} \abs{h}^{p+1} }{\abs{h}^{sp}} \frac{w(x) w(x+h)}{w(B_{x,x+h})} \,dh \,dx.
  \end{align*}
  With Lemma~\ref{lem:intweight} and $\lim_{s \to 1} c_s =0$ we estimate
  \begin{align*}
    \mathrm{I}_{1,2} &\lesssim c_s \norm{\nabla^2 v}_{L^\infty} \norm{\nabla v}_{L^\infty}^{p-1}
    \int_{B_{r+1}(0)} \int_{B_1(0)} \abs{h}^{p+1-sp} \frac{ w(x+h)}{w(B_{x,x+h})} \,dh\, w(x) \,dx
    \\
    &\lesssim \frac{c_s}{p+1-sp} \norm{\nabla^2 v}_{L^\infty} \norm{\nabla v}_{L^\infty}^{p-1}  w(B_{r+1}(0)) \xrightarrow{s \to 1} 0.
  \end{align*}
  Now, with Lemma~\ref{lem:doubling-annulus}
  \begin{align*}
    \mathrm{I}_{1,1} &\eqsim
    c_s \int_{B_{r+1}(0)} \underbrace{\int_{B_1(0)} \bigg( \frac{\abs{\skp{\nabla v(x)}{h}}}{\abs{h}^s} \bigg)^p \frac{w(x+h)}{w(B_{\abs{h}}(x))} \,dh}_{\coloneqq g_s(x)} \,w(x) \,dx.
  \end{align*}
  Again with Lemma~\ref{lem:doubling-annulus} we estimate
  \begin{align*}
    \abs{g_s(x)} &\lesssim \abs{\nabla v(x)}^p \int_{B_1(0)} \abs{h}^{(1-s)p} \frac{w(x+h)}{w(B_\abs{h}(x))} \,dh \eqsim \frac{1}{(1-s)p} \abs{\nabla v(x)}^p.
  \end{align*}
  Let $S$ denote the sector $S \coloneqq \set{h \,:\, \skp{\nabla v(x)}{h} \geq \frac 12 \abs{\nabla v(x)} \abs{h}}$. Then
  \begin{align*}
    \abs{g_s(x)} &\gtrsim \abs{\nabla v(x)}^p \int_{S \cap B_1(0)} \abs{h}^{(1-s)p} \frac{w(x+h)}{w(B_\abs{h}(x))} \,dh \eqsim \frac{1}{(1-s)p} \abs{\nabla v(x)}^p,
  \end{align*}
  where the last step requires an argument similar to Lemma~\ref{lem:intweight} but for sectors. Overall, we have
  \begin{align*}
    \mathrm{I}_{1,1} &\eqsim
    \frac{c_s}{(1-s)p} \int_{B_{r+1}(0)} \abs{\nabla v(x)}^p w(x)\,dx  \xrightarrow{s \to 1}  \frac 1p\!\!\! \int_{B_{r+1}(0)} \abs{\nabla v(x)}^p w(x)\,dx .
  \end{align*}
  Combing all estimates proves the claim for $s \nearrow 1$.

  Let us consider the case $s \searrow 0$. Then we have $ \mathcal{J}^{s}_{p,w}(v) \eqsim \mathrm{II}_1 + \mathrm{II}_2$, where
  \begin{align*}
    \mathrm{II}_1 &\coloneqq
    c_s \int_{2B} \int_{2B} \bigg(\frac{|v(x)-v(y)|}{|x-y|^s}\bigg)^p\frac{w(x)w(y)}{w(B_{x,y})}\,dy\,dx,
    \\
    \mathrm{II}_2 &\coloneqq
    c_s \int_{B} \int_{(2B)^c}\bigg(\frac{|v(y)|}{|x-y|^s}\bigg)^p\frac{w(x)w(y)}{w(B_{x,y})}\, dy\,dx.
  \end{align*}
  $\mathrm{II}_1$ vanishes as $s\searrow 0$. Indeed, with Lemma~\ref{lem:intweight} we can estimate
  \begin{align*}
    \mathrm{II}_1 &\lesssim c_s \norm{\nabla v}_\infty^p \int_{2B} \int_{2B} \abs{x-y}^{(1-s)p} \frac{w(x)}{w(B_{x,y})}\,dx \,w(y)\,dy
    \\
    &\lesssim c_s \norm{\nabla v}_\infty^p \frac{r^{(1-s)p}}{(1-s)p} w(2B) \xrightarrow{s\to 0} 0.
  \end{align*}
  Using a modified version of Lemma~\ref{lem:intweight} we get
  \begin{align*}
    \mathrm{II}_2 &\eqsim
    c_s \int_{B} \abs{v(y)}^p w(y) \int_{(2B)^c} |x-y|^{-sp}\frac{w(x)}{w(B_{x,y})}\,dx\,dy
    \\
    &\eqsim
    c_s \int_{B} \abs{v(y)}^p w(y) \frac{r^{-sp}}{sp}\,dy
    \;\;\xrightarrow{s\to 0} \;\;
     \frac 1p \int_{B} \abs{v(y)}^p w(y) \,dy.
   \end{align*}
   This proves the case~$s \searrow 0$.
\end{proof}

\begin{remark}\label{rem:doubling}
  All statements from Section \ref{ssec:basic-embedding} and Section \ref{ssec:limit-s} remain valid if the Muckenhoupt condition $w\in A_p$ is replaced by the weaker assumption that $w$ is doubling. In this case the dependence of the (implicit) constants would no longer be on $[w]_{A_p}$, but instead on the doubling constant $c_w$ (not necessarily in the same way). Indeed, $w\in A_p$ is used explicitly only in the proof of Lemma \ref{lem:doubling-annulus} which also holds with doubling weights (but with constants depending on $c_w$ in a nonlinear fashion).
\end{remark}

\subsection{Poincaré inequality}\label{sec:poincare}
In this section we establish a Poincaré type estimate for our function spaces. Our estimate is stable when $s\nearrow 1$.
For $0<\alpha<n$ the Riesz potential of~$f\,:\, \RRn \to \RR$ is defined as
\begin{align*}
  \mathcal{I}_\alpha f(x) \coloneqq \int_{\RRn} f(y) \abs{x-y}^{\alpha-n}\,dy.
\end{align*}
If $f$ is defined only on some $\Omega\subset \RRn$, we denote $\mathcal{I}_\alpha f(x)=\mathcal{I}_{\alpha}\bar f (x)$, where $\bar f$ is the zero extension of $f$ to $\RRn$.

To avoid proving similar estimates for the mollifiers now and for \Poincare{}'s inequality later, we prove two auxiliary lemmas. We first prove an elementary estimate for the Riesz potential.

\begin{lemma}\label{lem:riesz-weighted-basic}
  Let $\alpha\in(0,n)$, $1<p<\infty$, $w\in A_p$ and let $B_r\subset \RRn$ be a ball with radius~$r$. Then for every $f\in L_w^p(B_r)$ we have
  \begin{align*}
    \bigg(\frac{1}{w(B_r)}\int_{B_r} \!\! \big(\mathcal{I}_\alpha|f|\big)^pw\,dx\! \bigg)^{\frac 1p}  \! \lesssim \frac{r^\alpha}{\alpha} \left(\frac{1}{w(B_r)}\int_{B_r}\!|f|^pw\,dx\!\right)^{\!\frac{1}{p}},
  \end{align*}
  where the hidden constant depends continuously on $p$ and $n$ and $[w]_{A_p}$.
\end{lemma}

\begin{proof}
  By \eqref{eq:w-vs-sigma-B}, H\"older's inequality and Lemma~\ref{lem:intweight} ,
  \begin{align*}
    \int_{B_r} \!\! \big(\mathcal{I}_\alpha|f|\big)^pw\,dx
    & \eqsim \int_{B_r} \bigg(\int_{B_r} \frac{\abs{f(y)}|x-y|^\alpha w(y)^{1/p}\sigma(y)^{1/p'}}{w(B_{x,y})^{1/p} \sigma(B_{x,y})^{1/p'}}\,dy\bigg)^p w(x)\,dx  \\
    & \le  \int_{B_r}\bigg(\int_{B_r} \frac{\abs{f(y)}^p |x-y|^\alpha w(y)}{w(B_{x,y})}\,dy\bigg) \bigg(\int_{B_r} \frac{|x-y|^\alpha \sigma(y)}{\sigma(B_{x,y})}\,dy\bigg)^{\frac{p}{p'}} w(x)\,dx \\
    & \lesssim \bigg(\frac{r^{\alpha}}{\alpha}\bigg)^{\frac{p}{p'}} \int_{B_r} \abs{f(y)}^p  \bigg(\int_{B_{2r}(y)} \frac{|x-y|^\alpha w(x)}{w(B_{x,y})}\,dx\bigg)  w(y)\,dy \\
    & \eqsim \bigg(\frac{r^{\alpha}}{\alpha}\bigg)^{p} \int_{B_r} \abs{f(y)}^p   w(y)\,dy.
  \end{align*}
  This finishes the proof.
\end{proof}

\begin{lemma}
  \label{lem:aux2}
  Let $p\in (1,\infty)$ and $w \in A_p$. Let $B$ be a ball of radius~$r>0$ and $\psi\in L^\infty(B)$ be non-negative with $\norm{\psi}_{L^1(B)}=1$ and $\norm{\psi}_{L^\infty(B)}\leq c_0r^{-n}$.  Then for every $\alpha \in (0,s)$, $x \in B$ and $v\in W^{s,p}_w (B)$ we have
  \begin{align*}
    \!\!\!\!\abs{v(x)\!-\! \mean{v}_\psi}
    &\lesssim \frac{(1\!-\!\alpha)r^{s-\alpha}}{(s-\alpha)^{\frac 1{p'}}}  \int_{B}\bigg[\int_{B}
    \frac{\abs{v(y)-v(z)}^p}{\abs{y-z}^{sp}} \frac{w(y)\,dy}{w(B_{y,z})}
    \bigg]^{\frac 1p}
    \,\frac{dz}{\abs{x-z}^{n-\alpha}},
  \end{align*}
  where the hidden constant depends continuously on $n$, $p$, $c_0$ and $[w]_{A_p}$.
\end{lemma}

This lemma is stable when $s\nearrow 1$ in the sense that for $s>\tfrac 12$ we can set $1-\alpha \coloneqq 2(1-s)$ and thereby obtain a factor of $(1-s)$ in front of the right hand side. To achieve this stability we apply our fractional Riesz type estimate (see Appendix \ref{sec:appendix}).

\begin{proof}[Proof of Lemma \ref{lem:aux2}]
  We obtain by Lemma~\ref{lem:Rieszpotential}
  \begin{align*}
    \mathrm{I} \coloneqq \abs{v(x)\!-\! \mean{v}_\psi}
    &\lesssim (1\!-\!\alpha) \int_{B}\!\int_{B}\!
    \frac{\abs{v(y)\!-\!v(z)}\,dy}{\abs{y-z}^{n+\alpha}} \frac{dz}{ \abs{x-z}^{n-\alpha}}.
  \end{align*}
  Using $w^{\frac 1p} \sigma^{\frac 1{p'}}=1$ and~\eqref{eq:w-vs-sigma-B} and  H\"older's inequality we obtain
  \begin{align*}
    \mathrm{I}
    &\lesssim (1\!-\!\alpha) \int_{B}\bigg[\int_{B}
    \frac{\abs{v(y)-v(z)}^p}{\abs{y-z}^{sp}} \frac{w(y)\,dy}{w(B_{y,z})}
    \bigg]^{\frac 1p}
    \\
    &\qquad \quad\qquad \cdot \bigg[ \int_{B}
    \frac{\sigma(y)}{\sigma(B_{y,z})}
    \frac{dy}{\abs{y-z}^{(\alpha-s)p'}} \bigg]^{\frac 1{p'}} \,\frac{dz}{\abs{x-z}^{n-\alpha}}.
  \end{align*}
  Using Lemma~\ref{lem:intweight} we get
  \begin{align*}
    \mathrm{I}
    &\lesssim \frac{(1\!-\!\alpha)r^{s-\alpha}}{(s-\alpha)^{\frac 1{p'}}}  \int_{B}\bigg[\int_{B}
    \frac{\abs{v(y)-v(z)}^p}{\abs{y-z}^{sp}} \frac{w(y)\,dy}{w(B_{y,z})}
    \bigg]^{\frac 1p}
    \,\frac{dz}{\abs{x-z}^{n-\alpha}}.
  \end{align*}
  This proves the claim.
\end{proof}

We are now ready to prove a version of \Poincare 's inequality which is stable when $s \nearrow 1$.

\begin{theorem}[\Poincare]
  \label{thm:poincare}
  Let $B\subset \RRn$ be a ball, $p\in(1,\infty)$, $w \in A_p$ and $v \in W^{s,p}_w (B)$. Then
  \begin{align*}
    \int_B \abs{v-\mean{v}_B}^p w(x) dx
    \lesssim (1-s) r^{sp} \int_{B} \int_{B} \frac{\abs{v(x)-v(y)}^p}{\abs{x-y}^{sp}} \frac{w(x)w(y)}{w(B_{x,y})} dx\,dy ,
  \end{align*}
  where the implicit constant depends continuously on $n$, $p$ and $[w]_{A_p}$. Note that the mean value $\mean{v}_B$ on the left hand side can be replaced by other mean values, see Remark \ref{rem:mean}.
\end{theorem}

\begin{proof}
  Start by assuming $s \in [\frac23,1)$. We begin by applying
  Lemma \ref{lem:aux2} and Lemma \ref{lem:riesz-weighted-basic} to get
  \begin{align*}
    &\int_B \abs{v-\mean{v}_B}^p w(x) dx \\
    &\lesssim \frac{(1-\alpha)^p}{(s-\alpha)^{p-1}}r^{(s-\alpha)p} \!
    \int_B\!\Bigg[ \int_B \bigg[ \frac{\abs{v(y)\!-\!v(z)}^p}{\abs{y-z}^{sp}} \frac{w(y)}{w(B_{y,z})} dy \bigg]^{\frac1{p}} \frac{dz}{\abs{x-z}^{n-\alpha}} \Bigg]^p\! w(x)\,dx \\
    &\lesssim \frac{(1-\alpha)^p}{\alpha^p(s-\alpha)^{p-1}}r^{sp} \!
    \int_B \int_B \frac{\abs{v(y)-v(x)}^p}{\abs{y-x}^{sp}} \frac{w(y)w(x)}{w(B_{y,x})} dx\,dy.
  \end{align*}
  By choosing $\alpha=2s-1$ we have $\frac{(1-\alpha)^p}{\alpha^p(s-\alpha)^{p-1}} \eqsim c_s$, giving the desired estimate. For $s \in (0,\frac23)$ on the other hand, using $w^{\frac 1p} \sigma^{\frac 1{p'}}=1$ and~\eqref{eq:w-vs-sigma-B} and  H\"older's inequality we obtain
  \begin{align*}
    & \abs{v(x)-\mean{v}_B}^p \\
    & \lesssim r^{-np}\int_B \frac{\abs{v(x)-v(y)}^p}{\abs{x-y}^{sp}} \frac{w(y)}{w(B_{x,y})} dy
    \bigg[ \int_B \abs{x-y}^{(n+s)p'}\frac{\sigma(y)}{\sigma(B_{x,y})} dy \bigg]^{\frac{p}{p'}}.
  \end{align*}
  Using Lemma \ref{lem:intweight} we get
  \begin{align*}
    \abs{v(x)-\mean{v}_B}^p \lesssim r^{sp}\int_B \frac{\abs{v(x)-v(y)}^p}{\abs{x-y}^{sp}} \frac{w(y)}{w(B_{x,y})} dy,
  \end{align*}
  and therefore
  \begin{align*}
    &\int_B \abs{v-\mean{v}_B}^p w(x) dx \lesssim r^{sp}\int_B\int_B \frac{\abs{v(x)-v(y)}^p}{\abs{x-y}^{sp}} \frac{w(y)w(x)}{w(B_{x,y})} dx\,dy.
  \end{align*}

  Since for this range of $s$ we have $s^{-1} c_s\eqsim 1$, this finishes the proof.
\end{proof}
We can use the decomposition technique of~\cite{DieningRuzickaSchumacher10} to extend our \Poincare{} inequality to \emph{John domains}. A domain~$\Omega \subset \RRn$ is called an $\alpha$-John domain if there exists $x_0 \in \Omega$ such for all $x \in \Omega$ there exists a path $\gamma$ from $x$ to $x_0$ within~$\Omega$ parameterized by its arclength such that the carrot $\mathrm{car}(x,\alpha)$ is contained in~$\Omega$, where $\mathrm{car}(x,\alpha) \coloneqq \bigcup_{t \in [0,\textrm{length}(\gamma)]} B_{t/\alpha}(\gamma(t))$.

\begin{corollary}
  \label{cor:poincare-john}
  Let $\Omega$ be a bounded $\alpha$-John domain, $1<p<\infty$, $w \in A_p$. Then
  \begin{align*}
    \int_\Omega \abs{v-\mean{v}_\Omega}^p w(x) dx \lesssim (1-s) r^{sp} \int_\Omega \int_\Omega \frac{\abs{v(x)-v(y)}^p}{\abs{x-y}^{sp}} \frac{w(x)w(y)}{w(B_{x,y})} dx\,dy,
  \end{align*}
where the implicit constant depends on $n$, $p$, $[w]_{A_p}$ and $\alpha$. The mean value $\mean{v}_\Omega$ on the left hand side can be replaced by other mean values, see Remark \ref{rem:mean}.
\end{corollary}
\begin{proof}
  It has been shown in~\cite{DieningRuzickaSchumacher10} that $\Omega$ can be decomposed into a locally finite, countable family of Whitney cubes~$Q_i$ such that
  \begin{align}
    \label{eq:john-equivalence}
    \norm{f - \mean{f}_\Omega}_{L^p_w(\Omega)} &\eqsim \bigg( \sum_i
    \norm{f - \mean{f}_{Q_i}}_{L^p_w(Q_i)}^p\bigg)^{\frac 1p}.
  \end{align}
  Indeed, this follows from the decomposition theorem \cite[Theorem~4.2]{DieningRuzickaSchumacher10} together with the duality argument in the beginning of \cite[Theorem~5.1]{DieningRuzickaSchumacher10}. Hence,
  \begin{align*}
    \int_\Omega \abs{v-\mean{v}_\Omega}^p w(x) dx &\lesssim
    \int_{Q_i} \abs{v-\mean{v}_{Q_i}}^p w(x) dx
    \\
    &\lesssim (1-s) r^{sp} \sum_j \int_{Q_j} \int_{Q_j} \frac{\abs{v(x)-v(y)}^p}{\abs{x-y}^{sp}} \frac{w(x)w(y)}{w(B_{x,y})} dx\,dy
    \\
    &\lesssim (1-s) r^{sp} \int_\Omega \int_\Omega \frac{\abs{v(x)-v(y)}^p}{\abs{x-y}^{sp}} \frac{w(x)w(y)}{w(B_{x,y})} dx\,dy.
  \end{align*}
  This proves the claim.
\end{proof}

\begin{remark}\label{rem:mean}
  For $w\in A_p$ and $v\in L^p_w(B)$ there holds
  \begin{align*}
    \int_B \abs{v-\mean{v}_B}^pw(x)\,dx &\eqsim \inf_{v_0\in \setR}\int_B \abs{v-v_0}^pw(x)\,dx \eqsim \int_B \abs{v-\mean{v}_\psi}^pw(x)\,dx,
  \end{align*}
  where either $\norm{\psi}_{L^\infty(B)}\lesssim \frac{1}{\abs{B}}$ and $\int_B \psi \,dx=1$, or $\psi(x)= \frac{w(x)}{w(B)}$. This allows us to replace the mean values in Theorem \ref{thm:poincare}, \ref{thm:sobolev_poincare} and Corollary \ref{cor:poincare-john}, \ref{cor:SobolevJohn} by other mean values.
\end{remark}

We conclude this section by showing $W^{1,p}_w(B)\embedding W^{s,p}_{w}(B)$.

\begin{lemma}\label{lem:poincare_above}
  Let $B\subset \RRn$ be a ball, $p \in (1,\infty)$, $w \in A_p$ and $v \in W^{1,p}_w(B)$. Then
  \begin{align*}
    \Jspw (v|B) \lesssim \int_B \abs{\nabla v}^p \,w(x)dx,
  \end{align*}
  where the implicit constant depends on $n$, $p$ and $[w]_{A_p}$.
\end{lemma}

\begin{proof}
  By \cite[Chapter 1]{MalyZiemerBook} we have for almost $x,y\in B$
  \begin{align*}
    \abs{v(x)-v(y)}\lesssim \abs{x-y} \left(M(\indicator_B\nabla u)(x)+M(\indicator_B\nabla u)(y)\right),
  \end{align*}
  where $M$ is the Hardy-Littlewood maximal operator. This enables us to estimate
  \begin{align*}
    &c_s \int_B\int_B \frac{\abs{v(x)-v(y)}^p}{\abs{x-y}^{sp}}\frac{w(x)w(y)}{w(B_{x,y})}dx\,dy
    \\
    &\qquad\lesssim c_s \int_B \int_B \frac{M(\indicator_B\nabla u)(x)^p+M(\indicator_B\nabla u)(y)^p}{\abs{x-y}^{sp-p}}\frac{w(x)w(y)}{w(B_{x,y})}dx\,dy
    \\
    &\qquad\lesssim c_s \int_B  M(\indicator_B\nabla v(x))^pw(x)\int_B \frac{w(y)}{\abs{x-y}^{sp-p}w(B_{x,y})}dy\,dx
    \\
    &\qquad \lesssim sr^{(1-s)p} \int_B \abs{\nabla v(x)}^p w(x)dx ,
  \end{align*}
  where the implicit constants depend on $p$, $n$ and $[w]_{A_p}$.
\end{proof}

Theorem \ref{thm:poincare} and Lemma \ref{lem:poincare_above} enable us to show embeddings on general bounded domains:

\begin{corollary}\label{corr:Wsp_Lpw}
  If $\Omega\subset \RRn$ is open and bounded, then
  \begin{align*}
     W^{1,p}_{w,0}(\Omega)\embedding W^{s,p}_{w,0}(\Omega) \embedding L^p_w(\Omega).
  \end{align*}
\end{corollary}

\begin{proof}
  Let $B$ be a ball with $\Omega \subset B$. Using Lemma \ref{lem:poincare_above} and Theorem~\ref{thm:poincare} we have
  \begin{align*}
    W^{1,p}_{w,0}(\Omega) \embedding W^{1,p}_{w,0}(2B) \embedding W^{s,p}_{w,0}(2B)\embedding L^p_w(2B).
  \end{align*}
  This finishes the proof.
\end{proof}

\subsection{Mollification}
\label{ssec:mollification}
In this section we investigate mollification on $\Wspw$. For this let $\rho \in W^{1,\infty}(\RRn)$ be a standard Lipschitz-mollifier in the sense that $\rho \geq 0$, $\support \rho \subset \overline{B_1(0)}$, $\norm{\rho}_1=1$. For $\epsilon>0$ let $\rho_\epsilon(x)\coloneqq \epsilon^{-n}\rho(x/\epsilon)$. Then
$\norm{\rho_\epsilon}_\infty + \epsilon \norm{\nabla \rho_\epsilon}_\infty \lesssim \epsilon^{-n}$.

It is well known that if $w \in A_p$, $1<p<\infty$ and $v \in \Lpw$, then
\begin{enumerate}
  \item $\norm{v*\rho_\epsilon}_\Lpw \lesssim \norm{v}_\Lpw$,
  \item $v*\rho_\epsilon \to v$ in $\Lpw$ for $\epsilon \to 0$.
\end{enumerate}
This follows, e.g., from \cite[III.2 Theorem 2]{Stein70} and the boundedness of the maximal operator. The same conclusion holds if we replace $\Lpw$ with $W^{1,p}_w(\RRn)$. The goal of this section is to extend these properties to~$\Wspw$.

  Next, we will show stability of the energy under mollification.

  \begin{lemma}\label{lem:energy_mollifier}
  Let $s\in(0,1)$, $p\in (1,\infty)$ and $w\in A_p$. Let $\rho_\epsilon$ be the standard Lipschitz mollifiers as before. Then for every~$\epsilon>0$ and $v\in \Wspw$ we have
  \begin{align*}
    \Jspw(v * \rho_\epsilon) &\lesssim \Jspw(v)
  \end{align*}
  where the hidden constant depends continuously on~$p$ and $[w]_{A_p}$.
  \end{lemma}
  \begin{proof}
    We calculate
 \begin{align}
    \label{eq:moll_energy_split}
    \begin{split}
    \lefteqn{\Jspw(v * \rho_\epsilon)} \qquad &
      \\
      &\; \eqsim
       c_s\int_{\RRn} \!\int_\RRn \indicator_{\set{\abs{x-y}\leq \epsilon}} \bigg(\frac{\abs{(v * \rho_\epsilon)(x)\!-\!(v * \rho_\epsilon)(y)}}{\abs{x-y}^s}\bigg)^p \frac{w(x)w(y)}{w(B_{x,y})} \,dx\,dy
      \\
      &\; +c_s\int_{\RRn} \!\int_\RRn \indicator_{\set{\abs{x-y}> \epsilon}} \bigg(\frac{\abs{(v * \rho_\epsilon)(x)\!-\!(v * \rho_\epsilon)(y)}}{\abs{x-y}^s}\bigg)^p \frac{w(x)w(y)}{w(B_{x,y})} \,dx\,dy
      \\
      &\eqqcolon \mathrm{I}+\mathrm{II}.
    \end{split}
  \end{align}
  We start with the term $\mathrm{I}$.
  Suppose that $\abs{x-y}\leq \epsilon$. Then using $\norm{\nabla \rho_\epsilon}_\infty \lesssim \epsilon^{-n-1}$ and $w \in A_p$ we obtain
  \begin{align*}
    \abs{(v*\rho_\epsilon)(x)\!-\!(v*\rho_\epsilon)(y)}
    &\lesssim \frac{\abs{x-y}}{\epsilon} \!\dashint_{B_{2\epsilon}(x)} \abs{v(z) \!-\! \mean{v}_{B_\epsilon(x)}}\,dz
    \\
    &\lesssim \frac{\abs{x-y}}{\epsilon} \bigg( \int_{B_{2\epsilon}(x)} \abs{v(z) - \mean{v}_{B_\epsilon(x)}}^p \frac{w(z)\,dz}{w(B_{2\epsilon}(x))}\bigg)^{\frac 1p}.
  \end{align*}
  We plug this into~$\mathrm{I}$ and apply Lemma~\ref{lem:intweight} to the integration over~$y$ to get
  \begin{align*}
    \mathrm{I}
    &\lesssim c_s \int_\RRn \int_\RRn \! \indicator_{\set{\abs{x-y}\leq \epsilon}} \frac{\abs{x\!-\!y}^{-sp+p}}{\epsilon^p} \!\!\! \int_{B_{2\epsilon}(x)} \!\!\!\abs{v(z)\!-\! \mean{v}_{B_\epsilon(x)}}^p  \frac{w(z)w(x)w(y)}{w(B_{2\epsilon}(x))w(B_{x,y})} \,dz\,dx\,dy
    \\
    &\lesssim c_s \frac{\epsilon^{-sp}}{p-sp} \int_\RRn \int_{B_{2\epsilon}(x)} \abs{v(z) - \mean{v}_{B_\epsilon(x)}}^p \frac{w(z) \,dz}{w(B_{2\epsilon}(x))} w(x)\,dx.
  \end{align*}
  First assume that $s < \frac23$. Then, by Jensen's inequality and $w(B_{z,\xi}) \lesssim w(B_\epsilon(x))$ we get
  \begin{align*}
    \mathrm{I}
    &\lesssim c_s \frac{\epsilon^{-sp}}{1-s} \int_\RRn \int_{B_{2\epsilon}(x)} \int_{B_{\epsilon}(x)} \abs{v(z) - v(\xi)}^p \frac{w(z)w(\xi)}{w(B_{\epsilon}(x))}d\xi\,dz \frac{w(x)}{w(B_{2\epsilon}(x))}dx \\
    &\lesssim c_s \frac{\epsilon^{-sp}}{1-s} \int_\RRn \int_\RRn \indicator_{\set{\abs{z-\xi}\leq 4\epsilon}} \abs{v(z) - v(\xi)}^p \frac{w(z)w(\xi)}{w(B_{z,\xi})}d\xi\,dz \int_{B_{2\epsilon}(z)} \frac{w(x)}{w(B_{2\epsilon}(x))}dx \\
    &\lesssim c_s \int_\RRn \int_\RRn \indicator_{\set{\abs{z-\xi}\leq 4\epsilon}} \frac{\abs{v(z) - v(\xi)}^p}{\abs{z-\xi}^{sp}} \frac{w(z)w(\xi)}{w(B_{z,\xi})}d\xi\,dz.
  \end{align*}
  Now for the case $s\geq \frac23$. Then we can use Lemma~\ref{lem:aux2} and Lemma~\ref{lem:riesz-weighted-basic} to get
  \begin{align*}
    \mathrm{I}
    &\lesssim
    \frac{c_s \epsilon^{-\alpha p}(1\!-\!\alpha)^p}{(1\!-\!s)(s\!-\!\alpha)^{p-1}}
    \\
    &\cdot\! \int_\RRn \int_{B_{2\epsilon}(x)} \!\!
    \Bigg[ \int_{B_{2\epsilon}(x)} \!\!\bigg[ \int_{B_{2\epsilon}(x)} \!\!\! \frac{\abs{v(y)\!-\!v(\zeta)}^p}{\abs{y\!-\!\zeta}^{sp}} \frac{w(y)}{w(B_{y,\zeta})}dy\bigg]^{\frac 1p} \!\! \frac{d\zeta}{\abs{y-\zeta}^{n-\alpha}}\Bigg]^p\! w(z) \,dz \frac{w(x)\,dx}{w(B_{2\epsilon}(x))}
    \\
    &\lesssim
    \frac{c_s (1\!-\!\alpha)^p}{\alpha^p(1\!-\!s)(s\!-\!\alpha)^{p-1}}
    \int_\RRn \int_{B_{2\epsilon}(x)} \!
    \int_{B_{2\epsilon}(x)} \!\!\! \frac{\abs{v(y)\!-\!v(z)}^p}{\abs{y\!-\!z}^{sp}} \frac{w(y)}{w(B_{y,z})}dy\,
    w(z) \,dz \frac{w(x)\,dx}{w(B_{2\epsilon}(x))}
    \\
    &\lesssim
    \frac{c_s (1\!-\!\alpha)^p}{\alpha^p(1\!-\!s)(s\!-\!\alpha)^{p-1}}
    \int_\RRn \int_\RRn \indicator_{\set{\abs{y-z}\leq 4 \epsilon}}
    \frac{\abs{v(y)\!-\!v(z)}^p}{\abs{y\!-\!z}^{sp}} \frac{w(y)w(z)}{w(B_{y,z})}dy\,dz.
  \end{align*}
  By choosing $\alpha = 2s-1$ we have $1-\alpha \eqsim 1-s \eqsim s-a$, leading to $\frac{c_s(1-\alpha)^p}{\alpha^p(1-s)(s-\alpha)^{p-1}}~\lesssim~c_s$.

  This gives the needed estimate. It remains to estimate $\mathrm{II}$.
  We define the ball
  \begin{align*}
    \widetilde{B}_{x,y} \coloneqq B_{\frac{\abs{x-y}}{4}}\left(\frac{x+y}{2}\right)=\tfrac 12 B_{x,y}.
  \end{align*}
  We then have
  \begin{align*}
    \abs{(v * \rho_\epsilon)(x)\!-\!(v * \rho_\epsilon)(y)}^p
    & \lesssim \abs{(v * \rho_\epsilon)(x)\!-\!\mean{v}_{\widetilde{B}_{x,y}}}^p
      + \abs{(v * \rho_\epsilon)(y)\!-\!\mean{v}_{\widetilde{B}_{x,y}}}^p \\
    & \eqqcolon \mathrm{II}_1 + \mathrm{II}_2.
  \end{align*}
  For $\mathrm{II}_1$, using $w \in A_p$, we calculate

  \begin{align*}
    \mathrm{II}_1
    & \lesssim \Biggabs{\int_\RRn \rho_\epsilon(x-z) \big(v(z)-\mean{v}_{\widetilde{B}_{x,y}}\big) \,dz}^p \\
    & \lesssim \int_{B_\epsilon(x)} \bigabs{v(z)-\mean{v}_{\widetilde{B}_{x,y}}}^p \frac{w(z)}{w(B_\epsilon(x))} \,dz \\
    & \lesssim \int_{B_\epsilon(x)} \int_{\widetilde{B}_{x,y}} \bigabs{v(z)-v(\xi)}^p \frac{w(z)w(\xi)}{w(B_\epsilon(x))w(\widetilde{B}_{x,y})} \,d\xi\,dz.
  \end{align*}
  The same estimate holds for $\mathrm{II}_2$ if we exchange the roles of $x$ and $y$. Plugging this into the original term $\mathrm{II}$ and using symmetry in $x$ and $y$ we get
  \begin{align*}
    \mathrm{II}
    & \lesssim 2c_s \int_\RRn \int_\RRn \int_{B_\epsilon(x)} \int_{\widetilde{B}_{x,y}} \bigabs{v(z)-v(\xi)}^p \frac{w(z)w(\xi)}{w(B_\epsilon(x))w(\widetilde{B}_{x,y})} \,d\xi\,dz \\
    & \qquad \quad \cdot \frac{\indicator_{\set{\abs{x-y}> \epsilon}}}{\abs{x-y}^{sp}} \frac{w(x)w(y)}{w(B_{x,y})} \,dx\,dy.
  \end{align*}
  Now, since $\abs{x-y} > \epsilon \geq \abs{x-z}$, and by the definition of $\xi$ we have $\abs{z-\xi} \lesssim \abs{x-y}$ and
  $w(B_{z,\xi}) \lesssim w(B_{x,y})$. We therefore conclude
  \begin{align*}
    \mathrm{II}
    & \lesssim c_s \int_\RRn \int_\RRn \int_{B_\epsilon(x)} \int_{\widetilde{B}_{x,y}} \frac{\abs{v(z)-v(\xi)}^p}{\abs{z-\xi}^{sp}} \frac{w(z)w(\xi)}{w(B_{z,\xi})} \,d\xi\,dz \\
    & \qquad \quad \cdot \indicator_{\set{\abs{x-y}> \epsilon}} \frac{w(x)w(y)}{w(B_\epsilon(x))w(\widetilde{B}_{x,y})} \,dx\,dy \\
    & \lesssim c_s \int_\RRn \int_\RRn \frac{\abs{v(z)-v(\xi)}^p}{\abs{z-\xi}^{sp}} \frac{w(z)w(\xi)}{w(B_{z,\xi})} \\
    & \qquad \quad \cdot \int_{B_\epsilon(z)} \int_\RRn \indicator_{\set{4\abs{x-\xi} \geq \abs{x-y} \geq \frac43 \abs{x-\xi}}} \frac{w(x)w(y)}{w(B_\epsilon(x))w(\widetilde{B}_{x,y})} \,dy\,dx\,d\xi\,dz.
  \end{align*}
  It remains to estimate both inner integrals. Using doubling properties of Muckenhoupt weights, it holds uniformly in $z$ and $\xi$ that
  \begin{align*}
    &
    \int_{B_\epsilon(z)} \int_\RRn \indicator_{\set{4\abs{x-\xi} \geq \abs{x-y} \geq \frac43 \abs{x-\xi}}} \frac{w(x)w(y)}{w(B_\epsilon(x))w(\widetilde{B}_{x,y})} \,dy\,dx \\
    & \qquad= \int_{B_\epsilon(z)} \frac{w(x)}{w(B_\epsilon(x))} \int_\RRn \indicator_{\set{4\abs{x-\xi} \geq \abs{x-y} \geq \frac43 \abs{x-\xi}}} \frac{w(y)}{w(\widetilde{B}_{x,y})} \,dy\,dx \\
    & \qquad\lesssim \int_{B_\epsilon(z)} \frac{w(x)}{w(B_\epsilon(z))} \,dx =1.
  \end{align*}
  Plugging this into the previous estimate leads to
  \begin{align*}
    \mathrm{II}
    & \lesssim c_s \int_\RRn \int_\RRn \frac{\abs{v(z)-v(\xi)}^p}{\abs{z-\xi}^{sp}} \frac{w(z)w(\xi)}{w(B_{z,\xi})} \,d\xi\,dz.
  \end{align*}
  This finishes the proof.
  \end{proof}

  Next we will establish a bound for the difference between a function and it's mollification.

\begin{lemma}
  \label{lem:conv-diff}
  Let $0<s<1$ and $w\in A_p$. Let $\rho_\epsilon$ be the standard Lipschitz mollifiers as before. Then for every~$\epsilon>0$ and $v\in \Wspw$ we have
  \begin{align*}
    \Jspw(v * \rho_\epsilon\!-\!v) &\lesssim (1-s)
    \int_{\RRn} \!\int_\RRn \indicator_{\set{\abs{x-y}\leq 4\epsilon}} \bigg(\frac{\abs{v(x)-v(y)}}{\abs{x-y}^s}\bigg)^p \frac{w(x)w(y)}{w(B_{x,y})} \,dx\,dy,
  \end{align*}
  where the hidden constant depends continuously on~$p$, $n$ and $[w]_{A_p}$.
  In particular, $\rho_\epsilon*v \to v $ in $\Wspw$ as $\epsilon \to 0$.
\end{lemma}

To make the estimate stable we would wish to produce a factor of $c_s=s(1-s)$ in front of the right hand side and thus our estimate seems to miss a factor of $s$. However, this factor can not be expected here, since the limiting case $s\searrow 0$ needs a tail term.

\begin{proof}[Proof of Lemma \ref{lem:conv-diff}]
  We calculate
  \begin{align*}
    \lefteqn{\Jspw(v * \rho_\epsilon\!-\!v)} \qquad &
      \\
      & \lesssim c_s\int_{\RRn}\!\int_{\RRn} \indicator_{\set{\abs{x-y}\leq \epsilon}}
      \bigg(\frac{\abs{v(x)-v(y)}}{\abs{x-y}^s}\bigg)^p \frac{w(x)w(y)}{w(B_{x,y})} \,dx\,dy
      \\
      &\; + c_s\int_{\RRn} \!\int_\RRn \indicator_{\set{\abs{x-y}\leq \epsilon}} \bigg(\frac{\abs{(v * \rho_\epsilon)(x)\!-\!(v * \rho_\epsilon)(y)}}{\abs{x-y}^s}\bigg)^p \frac{w(x)w(y)}{w(B_{x,y})} \,dx\,dy
      \\
      &\; +2c_s\int_{\RRn} \!\int_\RRn \indicator_{\set{\abs{x-y}> \epsilon}} \bigg(\frac{\abs{(v * \rho_\epsilon)(x)\!-\!v(x)}}{\abs{x-y}^s}\bigg)^p \frac{w(x)w(y)}{w(B_{x,y})} \,dx\,dy
      \\
      &\eqqcolon \mathrm{I}_1 + \mathrm{I}_2+\mathrm{I}_3.
  \end{align*}
  The term $\mathrm{I}_1$ has already the correct form.
  For the term $\mathrm{I}_2$ notice that it is identical with the term $\mathrm{I}$ appearing \eqref{eq:moll_energy_split}. The bound established in that proof also suffices here. It remains to bound $\mathrm{I}_3$.
  With Lemma~\ref{lem:intweight} and the fact $w \in A_p$ we estimate
  \begin{align*}
    \mathrm{I}_3 &\lesssim
    \frac{c_s\epsilon^{-sp}}{sp} \int_{\RRn}  \abs{(v * \rho_\epsilon)(x)\!-\!v(x)}^p w(x)\,dx \\
  &= \frac{c_s\epsilon^{-sp}}{sp} \int_{\RRn}  \Biggabs{\int_\RRn \rho_\epsilon(x-z)\big(v(z)\!-\!v(x) \big)\,dz}^p w(x)\,dx \\
  &\lesssim \frac{c_s\epsilon^{-sp}}{sp} \int_{\RRn}  \int_\RRn \indicator_{\set{\abs{x-z}\leq \epsilon}}\abs{v(z)-v(x)}^p \frac{w(z)}{w(B_\epsilon(x))}\,dz\,w(x)\,dx \\
  &\lesssim \frac{c_s}{sp} \int_{\RRn}  \int_\RRn \indicator_{\set{\abs{x-z}\leq \epsilon}} \frac{\abs{v(z)-v(x)}^p}{\abs{x-z}^{sp}} \frac{w(z)w(x)}{w(B_{x,z})}\,dz\,dx,
  \end{align*}
  where in the last line we used $\abs{x-z} \lesssim \epsilon$ as well as $w(B_{x,z}) \leq w(B_\epsilon(x))$.

Combining all estimates concludes the proof.\end{proof}

\subsection{Density of smooth functions}

In this section we will show the density of $C_c^\infty(\Omega)$ in $\Wspwz$.

\begin{theorem}\label{thm:density}
  Let $s\in (0,1)$, $p\in (1,\infty)$ and $w\in A_p$. Let $\Omega \subset \RRn$ be open, bounded with fat complement, i.e. there exists $\beta \in (0,1)$ such that for every $x\in \partial\Omega$ and $r>0$ we have
  \begin{align*}
    \abs{B_r(x)\setminus \Omega} \geq \beta \abs{B_r(x)}.
  \end{align*}
  Then $C_c^\infty (\Omega)$ is dense in $\Wspwz$.
\end{theorem}

\begin{proof}
  We split the proof in two steps. In the first step we will first show that every $v\in \Wspwz$
can be approximated by functions with compact support in $\Omega$. Afterwards, we show in the second step that we can make these functions smooth.
  \textbf{Step 1:} Fix $v\in \Wspwz$ and let $\epsilon >0$. We define the set
  \begin{align*}
    A_\epsilon \coloneqq \set{x\in \Omega : \distance (x,\Omega ^c)<\epsilon},
  \end{align*}
  and $\Omega_\epsilon \coloneqq \Omega \setminus A_\epsilon$. Let $\rho_{\epsilon }$ be the mollifier as in Section \ref{ssec:mollification}. Then $\psi_\epsilon \coloneqq \indicator_{\Omega_{\epsilon/2}}\ast \rho_{\epsilon/4}\in C_c^\infty (\Omega)$ fulfills
  \begin{align*}
    \indicator_{\Omega_\epsilon} \leq \psi_\epsilon \leq \indicator_{\Omega _{\epsilon/4}},\qquad \text{and} \qquad \abs{\nabla \psi_\epsilon} \lesssim \frac{1}{\epsilon}.
  \end{align*}
  Let $g_\epsilon \coloneqq v-v\psi_\epsilon$. Then $g_\epsilon = 0$ on $A_\epsilon^c$ and
  \begin{align}
    \label{eq:psieps}
    \abs{g_\epsilon(x)-g_\epsilon(y)}\lesssim \abs{v(x)-v(y)} + \min\Bigset{\frac{\abs{x-y}}{\epsilon},1}(\abs{v(x)}+\abs{v(y)}).
  \end{align}
  We will later show that $\Jspw (v-v\psi_\epsilon)=\Jspw (g_\epsilon)\to 0$. We estimate
  \begin{align*}
    \Jspw (g_\epsilon)&\leq \Jspw(g_\epsilon| A_{2\epsilon}\times A_{2\epsilon})+2\Jspw(g_\epsilon|  A_{2\epsilon}^c\times A_{\epsilon})\eqqcolon \mathrm{I}+\mathrm{II}.
  \end{align*}
  Then, by~\eqref{eq:psieps} and Lemma \ref{lem:intweight},
  \begin{align*}
    \mathrm{I} &\lesssim \Jspw(v|A_{2\epsilon}) +  \!\int_{A_{2\epsilon}}\!\int_{A_{2\epsilon}}\! \bigg(\min\biggset{\frac{\abs{x-y}^{1-s}}{\epsilon^{1-s}},\frac{\epsilon^s}{\abs{x-y}^s}}\frac{\abs{v(x)}}{\epsilon^s}\bigg)^p \dwxy
    \\
    &\lesssim \Jspw(v|A_{2\epsilon}) +\int_{A_{2\epsilon}} \frac{\abs{v(x)}^p}{\epsilon^{sp}} w(x)\,dx .
  \end{align*}
  We have, using $g_\epsilon (y)=0$ and $\abs{x-y}\geq \epsilon$ for $x\in A_\epsilon$ and $y\in A^c_{2\epsilon}$, and Lemma \ref{lem:intweight}
  \begin{align*}
    \mathrm{II}&\lesssim \int_{(A_{2\epsilon})^c}\int_{A_\epsilon}    \bigg(\frac{\abs{v(x)}}{\abs{x-y}^s}\bigg)^p\dwxy \lesssim \int_{A_\epsilon} \frac{\abs{v(x)}^p}{\epsilon^{sp}}w(x)\,dx
  \end{align*}
  Overall,
  \begin{align}\label{eq:aux_density}
    \Jspw (v-v\psi_\epsilon )\lesssim \Jspw(v|A_{2\epsilon}) +\int_{A_{2\epsilon}} \frac{\abs{v}^p}{\epsilon^{sp}} w\,dx .
  \end{align}
  We then find a family of balls $B_j$
  with midpoints $x_j \in \partial \Omega$ and radius $4\epsilon$ that cover $A_{2\epsilon}$, such that $\sum_j \indicator_{B_{j}}\leq c(n)$. By the fat complement condition on $\Omega$ we can use the \Poincare~inequality Theorem \ref{thm:poincare} in $B_j$  where we take the average $\mean{u}_{B_j\setminus \Omega}=0$. We then have
  \begin{align*}
    \int_{A_{2\epsilon}} \frac{\abs{v}^p}{\epsilon^{sp}} w \,dx &\lesssim \sum_j \int_{B_j} \frac{\abs{v}^p}{\epsilon^{sp}} w\,dx \lesssim \sum_j \frac 1s \Jspw (v |B_j)\lesssim \frac 1s \Jspw (v |\widetilde{A}_{4\epsilon} ),
  \end{align*}
  where $\widetilde{A}_{4\epsilon}\coloneqq\set{x: \distance (x,\partial \Omega)< 4\epsilon}$.
  Thus, by \eqref{eq:aux_density} we can conclude
  \begin{align*}
    \Jspw (v-v\psi_\epsilon) \lesssim \Jspw (v| \widetilde{A}_{4\epsilon}) \xrightarrow{\epsilon \to 0} 0.
  \end{align*}
  \textbf{Step 2:} By Step 1 we can find a function in $\Wspw$ with support compactly contained in $\Omega$, which is arbitrarily close to $v$ (in $\Wspwz$).
  Then, by Lemma \ref{lem:conv-diff} we can approximate this function by a $C_c^\infty (\Omega)$ function. This finishes the proof.
\end{proof}

\subsection{Sobolev-Poincaré inequality}
\label{sec:Sobolev_Poincare}

The following Lemma \ref{lem:riesz-weighted} is an improved version of Lemma \ref{lem:riesz-weighted-basic}.

\begin{lemma}\label{lem:riesz-weighted}
  Let $n\geq 2$, $s\in (0,1)$, $p\in (1,\infty)$, $w\in A_p$ and $\alpha_0\in(0,n)$. Then there exists $q<p$ such that for every $\alpha \in (0,\alpha_0)$, every ball $B\subset \RRn$ and $f\in L^p_w(B)$ we have
  \begin{align*}
    \bigg(\frac{1}{w(B)}\int_{B} \!\! \big(\mathcal{I}_\alpha|f|\big)^{\frac{np}{n-\alpha \frac pq}}w\,dx\! \bigg)^{\frac{n-\alpha \frac pq}{np}}  \! \lesssim  [w]_{A_p}^{\!\frac{np-\alpha}{nq(p-1)}} \frac{r^\alpha}{\alpha} \left(\frac{1}{w(B)}\int_{B}\!|f|^pw\,dx\!\right)^{\!\frac{1}{p}}.
  \end{align*}
  The hidden constant depends only on $n$, $p$ and $\alpha_0$.
\end{lemma}
\begin{proof}
  The proof can be found in~\cite[Theorem~1.1]{AlbericoCianchiSbordone09}. The additional restriction~$1<p<\frac n \alpha$ in their result is only used for the other parts of that theorem. We additionally keep track on the dependence of all constants on $\alpha$ and $p$ by providing some detail on how the constants $C_1,\dots,C_6$ in \cite{AlbericoCianchiSbordone09} depend on them: First we have $C_1 =2^n / \alpha$, $C_2=C(n,p)$, $C_3=C(n,p,\alpha_0)$ and we can choose $C_4=C(n,p,\alpha_0)/\alpha$. Then, replacing $(7)$ in \cite{AlbericoCianchiSbordone09} by \cite[Theorem 9.1.9]{Gramodern14} combined with \cite[Theorem 1.3.2]{Gra14}, we get $C_5 = C_4 (2\cdot 24^n\cdot p')^{(n-\alpha)/n}$. Finally, we can put $C_6 = 2C_5$. This finishes the proof.
\end{proof}
\begin{theorem}[Sobolev-\Poincare]
  \label{thm:sobolev_poincare}
  Let $n\geq 2$, $s\in (0,1)$, $p\in (1,\infty)$ and $w\in A_p$. For every ball $B\subset \RRn$ and $v \in W^{s,p}_w(B)$ it holds that
  \begin{align*}
    &\Bigg(\frac1{w(B)}\int_B \abs{v-\mean{v}_B}^\frac{np}{n-s} w\, dx\Bigg)^\frac{n-s}{n} \\
    &\quad \lesssim   s^{1-2p}(1-s) \frac{r^{sp}}{w(B)} \int_{B} \int_{B} \frac{\abs{v(x)-v(y)}^p}{\abs{x-y}^{sp}} \frac{w(x)w(y)}{w(B_{x,y})} dx\,dy,
  \end{align*}
  where the implicit constant depends on $n$, $p$ and $[w]_{A_p}$. Note that the mean value $\mean{v}_B$ on the left hand side can be replaced by other mean values, see Remark \ref{rem:mean}.
\end{theorem}
\begin{proof}
  Let $\alpha\in (0,s)$. We will choose the exact value of $\alpha$ later. For $z\in B$ we define the function $f$ by
  \begin{align*}
    f(z) \coloneqq \frac{1-\alpha}{(s-\alpha)^{1/p'}}\Bigg( \int_{B} \frac{\abs{v(z)-v(y)}^p}{\abs{z-y}^{sp}} \frac{w(y)}{w(B_{z,y})} dy\Bigg)^\frac1{p}.
  \end{align*}
  Then by Lemma~\ref{lem:aux2} we have
  \begin{align*}
    \abs{v(x)-\mean{v}_B} \lesssim r^{s-\alpha} \mathcal{I}_\alpha f(x).
  \end{align*}
  Combining this with Lemma~\ref{lem:riesz-weighted} implies that we can find a $q$ independent of $\alpha$ such that
  \begin{align}\label{eq:proofSP}
    \begin{aligned}
      \lefteqn{
      \bigg(\frac{1}{w(B)}\int_B \abs{v-\mean{v}_B}^{\frac{np}{n-\alpha \frac pq}} w dx\bigg)^{\frac{n-\alpha \frac pq}{n}}}\qquad&
      \\
      &\lesssim \frac{r^{sp}}{\alpha^p}\frac{1}{w(B)}\int_B f^p wdx
      \\
      &=\frac{(1-\alpha)^p}{\alpha^p(s-\alpha)^{p-1}}\frac{r^{sp}}{w(B)}\int_B \int_B \frac{\abs{v(x)-v(y)}^p}{\abs{x-y}^{sp}} \frac{w(x)w(y)}{w(B_{z,y})} \,dx\,dy .
    \end{aligned}
  \end{align}
  We could now choose $\alpha$ to be any value in $[\frac qp s,s)$ and \eqref{eq:proofSP} would imply the claimed inequality except for the precise dependence of the constant on $s$. To achieve this we have to choose $\alpha=\alpha(s) \in [\frac qp s , s)$ in such a way that
  \begin{align}\label{eq:proofSPlast}
    \frac{(1-\alpha)^p}{\alpha^p(s-\alpha)^{p-1}}\leq c (1-s)s^{1-2p},
  \end{align}
  for some constant $c$ independent of $s$ and $\alpha$. We distinguish two cases:
  If $s\leq(2-\frac qp)^{-1}$, then we set $\alpha = \frac qp s$. This yields $(1-s)\eqsim (1-\alpha)\eqsim 1$ and $\alpha \eqsim s-\alpha \eqsim s$. Therefore \eqref{eq:proofSPlast} holds.
  If $s>(2-\frac qp)^{-1}$, then we set $\alpha = 2s-1$. This gives us $1-\alpha \eqsim 1-s\eqsim s-\alpha$ and $\alpha \eqsim s \eqsim 1$. Thus \eqref{eq:proofSPlast} holds in this case, too.
\end{proof}
\begin{remark}\label{rem:SobolevPoincare}
  For every fixed weight $w$, it is possible to improve the exponent $np/(n-s)$ on the left hand side of the Sobolev-\Poincare~inequality in Theorem~\ref{thm:sobolev_poincare} slightly. This can be done choosing the parameter $\alpha$ in the proof above to be closer to $s$ than to $\frac qp s$. However, the exponent $np/(n-s)$ is optimal in the sense that it is the largest one that works for all $w\in A_p$.
\end{remark}

The Sobolev-\Poincare~inequality extends to bounded John domains. This can be shown exactly as in Corollary \ref{cor:poincare-john}.

\begin{corollary}\label{cor:SobolevJohn}
  Let $n\geq 2$, $s\in (0,1)$, $p\in (1,\infty)$, $w\in A_p$ and $\Omega$ be a bounded $\alpha$-John domain. Then for every $v \in W^{s,p}_w(B)$ we have
  \begin{align*}
    &\Bigg(\frac1{w(\Omega)}\int_\Omega \abs{v-\mean{v}_\Omega}^\frac{np}{n-s} w\, dx\Bigg)^\frac{n-s}{n} \\
    &\quad \lesssim   s^{1-2p}(1-s) \frac{r^{sp}}{w(\Omega)} \int_{\Omega} \int_{\Omega} \frac{\abs{v(x)-v(y)}^p}{\abs{x-y}^{sp}} \frac{w(x)w(y)}{w(B_{x,y})} dx\,dy,
  \end{align*}
  where the implicit constant depends on $n$, $p$, $[w]_{A_p}$ and $\alpha$. Note that the mean value $\mean{v}_B$ on the left hand side can be replaced by other mean values, see Remark \ref{rem:mean}.
\end{corollary}

\subsection{Compact embedding}
\label{sec:compactness}

In this section we show that $\Wspwz$ embeds compactly into $L^p_w(\Omega)$. In the non fractional case, it is well known that for $p\in (1,\infty)$ and $w\in A_p$
\begin{align*}
  \norm{v-v*\rho_\epsilon}_{L^p_w(\RRn)}\lesssim \epsilon \norm{\nabla v}_{L^p_w(\RRn)}.
\end{align*}
This follows from the pointwise estimate $\abs{v-v*\rho_\epsilon}\lesssim \epsilon M(\nabla v)$ and the boundedness of the maximal operator. The details can be found, e.g., in \cite[Lemma 1.50, Theorem 1.32]{MalyZiemerBook}. The following lemma is a fractional analog of this.

\begin{lemma}\label{lem:mollifier_Lp}
  Let $s\in (0,1)$, $p\in (1,\infty)$, $w\in A_p$ and $v\in \Wspwz$. Then we have
  \begin{align*}
    \int_\RRn \abs{v-v*\rho_\epsilon}^p w(x)\,dx \lesssim \frac{\epsilon^{sp}}{s}\Jspw(v) ,
  \end{align*}
  where the implicit constant depends continuously on $n$, $p$ and $[w]_{A_p}$.
\end{lemma}

\begin{proof}
  Let $B_j$ be a locally finite family of balls with radius $\epsilon$ that cover $\RRn$. Then
  \begin{align*}
    \int_\RRn \abs{v-v*\rho_\epsilon}^p w\,dx &\lesssim \sum_j \int_{B_j} \abs{v-v*\rho_\epsilon}^pw\,dx
    \\
    &\lesssim \sum_j \int_{B_j} \abs{v-\mean{v}_{2B_j}}^pw\,dx + \sum_j \int_{B_j} \abs{(v-\mean{v}_{2B_j})*\rho_{\epsilon}}^pw\,dx
    \\
    &\lesssim \sum_j \int_{2B_j} \abs{v-\mean{v}_{2B_j}}^p w\, dx
    \\
    &\lesssim \frac 1s \sum_j \epsilon^{sp} \Jspw(v\,|\,2B_j) \lesssim \frac{\epsilon^{sp}}{s}\Jspw(v),
  \end{align*}
  where in the second step we used the boundedness of the convolution in $L^p_w$ and in the fourth step we used \Poincare~inequality. This finishes the proof.
\end{proof}

We are now ready to show the compactness result.

\begin{theorem}\label{thm:compactness}
  Let $\Omega \subset \RRn$ be open and bounded, $s\in (0,1)$, $p\in (1,\infty)$ and $w\in A_p$. Then the embedding $\Wspwz \embedding L^p_w(\Omega)$ is compact.
\end{theorem}

\begin{proof}
  Let $v_m \weakto v$ in $\Wspwz$. We have to show that $v_m \to v$ in $L^p_w(\Omega)$. For $\epsilon >0$ we have by Lemma \ref{lem:mollifier_Lp}
  \begin{align*}
    \int_\Omega\! \abs{v_m\!-\!v}^p w\,dx &\lesssim  \int_\Omega\! \abs{v_m\!-\!v_m\!*\rho_\epsilon}^pw\,dx+\!\!\int_\Omega\! \abs{(v_m\!-\!v)*\rho_\epsilon}^pw\,dx+\!\!\int_\Omega\! \abs{v\!-\!v\!*\!\rho_\epsilon}^pw\,dx
    \\
    &\lesssim \frac{\epsilon^{sp}}{s}\Jspw(v_m)+\int_\Omega \abs{(v_m-v)*\rho_\epsilon}^pw\,dx +\frac{\epsilon^{sp}}{s}\Jspw(v).
  \end{align*}
  By the embedding $\Wspwz \embedding L^p_w(\Omega)$ (Corollary \ref{corr:Wsp_Lpw}) we have $v_m \weakto v$ in $L^p_w(\Omega)$. This implies that $\abs{(v_m-v)*\rho_\epsilon}\xrightarrow{m}0$
  everywhere. Note that
  \begin{align*}
    \abs{(v_m-v)*\rho_\epsilon(x)}\lesssim \epsilon^{-n} \sup_m \norm{v_m-v}_{L^p_w(\Omega)} \norm{\indicator_{\widetilde{\Omega}}}_{L^{p'}_{w^{1-p}}}<\infty,
  \end{align*}
  where $\widetilde{\Omega} = \set{x \in \RRn\,:\, d(x,\Omega) \leq \epsilon}$.
  Thus, by dominated convergence, we have
  \begin{align*}
    \limsup_{m\to \infty} \int_\Omega\! \abs{v_m\!-\!v}^p w\,dx &\lesssim  \frac{\epsilon^{sp}}{s}\left(\limsup_{m\to \infty} \Jspw (v_m)+\Jspw(v)\right).
  \end{align*}
  Since $v_m \weakto v$ in $\Wspwz$ we have $\sup_m\Jspw (v_m)<\infty$. Thus, letting $\epsilon \to 0$, we have shown that $v_m \to v$ in $L^p_w(\Omega)$.
\end{proof}

\section{Regularity theory for degenerate nonlocal problems}
\label{sec:regularity}

In this section we study minimizers of the energy \eqref{eq:Jspw} with degenerate weights. We define the concept of weak (sub-)solutions and show local boundedness, Hölder continuity, and a nonlocal Harnack inequality.

For this section we fix  $0< s_0 \le s<1$, $p\in (1,\infty)$ and a weight $w\in A_p$.
\subsection{Degenerate nonlocal problems}\label{sec:degenerate_problems}

We now define nonlocal problems with degenerate weights. Recall from \eqref{eq:Jspw} that for $v \in L^1_{\loc}(\RRn)$ we set
\begin{align*}
  \Jspw (v)=c_s \int_\RRn \int_\RRn \bigg(\frac{|v(x)-v(y)|}{|x-y|^s}\bigg)^pk(x,y)\,dx\,dy ,
\end{align*}
where $k$ satisfies \eqref{eq:cond-k}. We note that  $k(x,y)=0$ when $w(x)=0$ or $w(y)=0$, hence the coefficient function $k$ can be degenerate. A typical example of $k$ can be seen in Example~\ref{example:a} below.

We say that $u\in \Wspw$ minimizes $\Jspw$ on $\Omega$ if $ \Jspw (u)\leq \Jspw(v)$ for all  $v\in \Wspw$ with $u=v$ in $\Omega^c$.

Note that even though $u$ minimizes $\Jspw$ only on $\Omega$, the nonlocal nature of our model makes it necessary that $u$ is defined on all of $\RRn$. It is possible to change the domains of integration in \eqref{eq:Jspw} from $\RRn \times \RRn$ to the smaller set $(\Omega^c \times \Omega^c)^c$. We will expand further on this in Remark \ref{rem:C_Omega}.

The following theorem asserts that for given complement data $g$ a unique minimizer of $\Jspw$ exists.

\begin{theorem}[Existence of minimizer]\label{thm:existence}
  Let $g \in \Wspw$. There exists a unique minimizer $u$ of $\Jspw$ on~$\Omega$ with complement data~$g$, i.e. it is the minimizer among all $v\in \Wspw$ with $v=g$ in $\Omega^c$.
\end{theorem}

\begin{proof}
  Set $M\coloneqq \min_{v\in \Wspwz} \Jspw (v+g)$. Note that $M>0$ and pick a sequence $v_k \in \Wspwz$ such that $\Jspw (v_k+g)\searrow M$ as $k\rightarrow \infty$. Since
  \begin{align*}
    \abs{v_k}_{\Wspw} \leq \left(\Jspw(v_k+g)\right)^{\frac 1p} + \abs{g}_{\Wspw},
  \end{align*}
  we have that $v_k$ is bounded in $\Wspwz$. By Lemma \ref{lem:Banachspace} we can pick a subsequence that converges weakly to some $v^* \in \Wspwz$. By the compact embedding (Lemma \ref{thm:compactness}) we can pick a further subsequence which converges to $v^*$ in $L^p_w(\Omega)$. From this we can extract a final subsequence that converges to $v^*$ pointwise almost everywhere. Let us thus assume without loss of generality that $v_k \rightarrow v$ pointwise almost everywhere. Using Fatou's Lemma, this implies
  \begin{align*}
    M \leq \Jspw (v^*+g)\leq \liminf_{k\rightarrow \infty} \Jspw (v_k+g) =M.
  \end{align*}
  Thus $u\coloneqq v^*+g$ is a minimizer. The uniqueness follows by the strict convexity from Lemma \ref{lem:Banachspace}.
\end{proof}

We notice that if  $u\in \Wspw$ is the unique minimizer from Theorem \ref{thm:existence}, then it satisfies
\begin{align}\label{weakform}
  \int_{\RRn} \int_{\RRn} k(x,y)\frac{|u(x)-u(y)|^{p-2}(u(x)-u(y))(\eta(x)-\eta(y))}{|x-y|^{sp}}\,dx\, dy =0
\end{align}
for every $\eta  \in \Wspwz$. This follows directly from the equation $\frac{d}{dt}  \Jspw (u+t \eta)\big| _{t=0}=0$. Moreover, if $u\in \Wspw$ satisfies  \eqref{weakform}, then it is the minimizer of $\Jspw$ with respect to the complement data $g$.

Note that by symmetry, \eqref{weakform} is equivalent to the equation
 $$
 \int_{\Omega} \bigg(\,\int_{\RRn} \frac{|u(x)-u(y)|^{p-2}(u(x)-u(y))}{|x-y|^{sp}}\frac{k(x,y)}{w(x)}\, dy \bigg) \eta(x)w(x)\,dx =0,
$$
if the inner integral is finite for almost every $x\in \Omega$. Then the last integral implies that for almost every $x\in \Omega$,
$$
\mathcal L ^{s}_{p,w}  u(x) \coloneqq  c_s\int_{\RRn} \frac{|v(x)-v(y)|^{p-2}(u(x)-u(y))}{|x-y|^{sp}}\frac{k(x,y)}{w(x)}\, dy =0.
$$
Note that $\lambda \frac{w(y)}{w(B_{x,y})}\le\frac{k(x,y)}{w(x)} \le \Lambda \frac{w(y)}{w(B_{x,y})}$. Therefore, \eqref{weakform} can be regarded as a weak form of the following degenerate nonlocal integro-differential equation:
\begin{equation}\label{PDE}
\mathcal L ^{s}_{p,w}  u  =0  \quad \text{in }\ \Omega.
\end{equation}

\begin{definition}
We say that $u\in \Wspw$ is a weak sub-(super-)solution to \eqref{PDE} if
\begin{equation}\label{weakform1}
 \int_{\RRn} \int_{\RRn} k(x,y)\frac{|u(x)-u(y)|^{p-2}(u(x)-u(y))(\eta(x)-\eta(y))}{|x-y|^{sp}}\,dx\, dy \le (\ge)0
\end{equation}
for every $\eta  \in \WspwO$ with $\eta\ge 0$ a.e. in $\RRn$. If $u$ is both a weak subsolution and a weak supersolution, the $u$ is called a weak solution.
\end{definition}
Note that if $u$ is a weak solution to \eqref{PDE}, then it satisfies \eqref{weakform}.

\begin{example}\label{example:a}
  Let $w(z) = \abs{z}^\gamma$ with $\gamma \in (-n,n(p-1))$. Then $w \in A_p$. We claim that
  \begin{align}
    \label{eq:ex1}
    k(x,y) &= \frac{\abs{x}^\gamma \abs{y}^\gamma}{(\abs{x}+\abs{y})^\gamma} \abs{x-y}^{-n}
  \end{align}
  with $x,y \in \RRn$ satisfies the assumption~\eqref{eq:cond-k}. To verify this, we have to show that
  \begin{align}
    \label{eq:ex2}
    w(B_{x,y}) \eqsim (\abs{x}+\abs{y})^\gamma \abs{x-y}^n.
  \end{align}
  To prove this we distinguish two cases. Firstly, suppose that $0 \notin 2B_{x,y}$. Then $\frac 16 (\abs{x}+\abs{y}) \leq\abs{z} \leq \abs{x}+\abs{y}$ for all $z \in B_{x,y}$  and therefore
  \begin{align*}
    w(B_{x,y}) \eqsim \int_{B_{x,y}} (\abs{x}+\abs{y})^\gamma \,dz \eqsim (\abs{x}+\abs{y})^\gamma \abs{x-y}^n.
  \end{align*}
  This proves~\eqref{eq:ex2} for $0 \notin 2B_{x,y}$.  Secondly, suppose that $0 \in 2B_{x,y}$. Then $\abs{x-y} \leq \abs{x}+\abs{y} \leq 3 \abs{x-y}$. Using that~$w$ is doubling, $B_{x,y} \subset 3B_{\abs{x-y}}(0)$ and $B_{\abs{x-y}}(0)\subset 4 B_{x,y}$ we calculate
  \begin{alignat*}{3}
    w(B_{x,y}) &\eqsim \int_{B_{\abs{x-y}}(0)} \abs{z}^\gamma\,dz &&\eqsim \abs{x-y}^{n+\gamma} &&\eqsim (\abs{x}+\abs{y})^\gamma \abs{x-y}^n.
  \end{alignat*}
  This proves~\eqref{eq:ex2} for $0 \in 2B_{x,y}$. Overall, we have proved that $w$ and $k$ satisfy~\eqref{eq:cond-k}.
\end{example}
For simplicity, we denote
$$
d\mu(x) = w(x)\,dx ,
\quad\text{and}\quad
\fint_U f d \mu = \frac{1}{\mu(U)} \int_U f\, d \mu,
$$
and define the nonlocal tail term with respect to the measure $\mu$ by
\begin{equation}\label{tail}
\mathrm{Tail}(f;\rho) = \mathrm{Tail}(f;x_0,\rho) \coloneqq \bigg((1-s)\rho^{sp} \int_{\RRn \setminus B_{\rho}(x_0)}\frac{|f(x)|^{p-1}}{|x-x_0|^{sp}}\frac{d\mu(x)}{w(B_{x,x_0})}\bigg)^{\frac{1}{p-1}}.
\end{equation}
Note that if $f \in \Wspw$ then $\mathrm{Tail}(f;x_0,\rho)<\infty$ at every $x_0$ and $\rho$ such that $B_{\rho}(x_0)\subset \Omega$.  Moreover, if $s\nearrow 1$ then $\mathrm{Tail}(f;x_0,\rho)$ vanishes.

We further recall the following density property of $A_p$ weights (see \cite[Theorem 7.2.7]{Gra14}): For every ball $B\subset \RRn$ and every  measurable $D\subset B$,
\begin{equation}\label{Apinequality}
  \frac{1}{[w]_{p}} \left( \frac{|D|}{|B|}\right)^p \le \frac{w(D)}{w(B)} \le  c \left( \frac{|D|}{|B|}\right)^{\sigma}
\end{equation}
for some $c,\sigma>0$ depending on $p$ and $[w]_{A_p}$.

Let us state three fundamental regularity results for the weak solution to \eqref{PDE}. They are the main results of the regularity part of this paper.

\begin{theorem}[Local boundedness]
\label{thm:bounded}
Let $u \in \Wspw$ be a weak subsolution to \eqref{PDE}. For every $\delta>0$ and $B_{r}=B_r(x_0) \Subset \Omega$, there holds
\begin{align}\label{lb}
\underset{B_{r/2}}{\mathrm{ess\,sup}}\, u \le c_b \delta^{-\frac{p-1}{p}\frac{n}{s_0}} \bigg(\fint_{B_{r}}u_+^p\,d\mu \bigg)^{1/p} + \delta\, \mathrm{Tail}(u_{+};x_0,r/2)
\end{align}
for some $c_{b}= c_{b}(n,p,\lambda,\Lambda,s_0,[w]_{A_p})>0$.
Moreover, if $u$ is a weak solution to \eqref{PDE}, then we have
\begin{equation}\label{lb1}
\| u \|_{L^\infty(B_{r/2})}  \le c_b \delta^{-\frac{p-1}{p}\frac{n}{s_0}} \bigg(\fint_{B_{r}} |u|^p\,d\mu \bigg)^{1/p} + \delta\,   \mathrm{Tail}(u;x_0,r/2),
\end{equation}
and thus $u\in L^\infty_{\mathrm{loc}}(\Omega)$.
\end{theorem}

\begin{theorem}[Hölder continuity]
\label{thm:Holder}
Let $u \in \Wspw$ be a weak solution to \eqref{PDE}. Then $u\in C^{\alpha}_{\loc}(\Omega)$ for some $\alpha\in (0,1)$, and for every $B_{r}=B_{r}(x_0)\Subset\Omega$, there holds
\begin{align*}
[u]_{C^{\alpha}(B_{r/2})} \le c\, r^{-\alpha} \bigg\{\bigg(\fint_{B_{r}} |u|^p\,d\mu \bigg)^{1/p} +   \mathrm{Tail}(u;x_0,r/2))\bigg\},
\end{align*}
where $c= c(n,p,\lambda,\Lambda,s_0,[w]_{A_p})>0$.
\end{theorem}

\begin{theorem}[Harnack's inequality]
\label{thm:Harnack}
Let  $u \in \Wspw$ be a weak solution to \eqref{PDE}. If $u\ge0 $ in some $B_{2r}=B_{2r}(x_0)\Subset\Omega$, then there holds
\begin{equation}\label{eq:Harnack}
\sup_{B_r} \le c \left( \inf_{B_r}u +   \mathrm{Tail}(u_-;x_0,r)\right)
\end{equation}
for some $c= c(n,p,\lambda,\Lambda,s_0,[w]_{A_p})>0$.
\end{theorem}

In the following subsections, we will prove the three theorems above.

\begin{remark}\label{rem:C_Omega}
  In Theorem \ref{thm:bounded}, Theorem \ref{thm:Holder}, and Theorem \ref{thm:Harnack} the assumption $u\in \Wspw$ can be weakened. It is not necessary to assume any regularity on the complement data $g=u|_{\overline{\Omega} ^c}$ other than some weighted integrability. Instead, it is sufficient to assume that $\Jspw ((\Omega^c \times \Omega ^c)^c)<\infty$, which also implies that  $\operatorname{Tail}(u;x_0,r)<\infty$ for every $B_r(x_0)\subset \Omega$. This approach has been introduced in \cite{ServValdinoci12,FelsKassmannVoigt15}. It is also possible to give an intrinsic definition of the complement data $g$ via nonlocal trace spaces defined only on $\Omega^c$. For this approach see \cite{GrubeHensiek24,GrubeKassmann23}. Our results remain valid with all those adaptions.
\end{remark}

\subsection{Local boundedness}

In this section we prove Theorem \ref{thm:bounded}.
We start with a Caccioppoli type estimate.
\begin{proposition}
\label{lem.caccio}
Let $u\in \Wspw$ is a weak subsolution  to \eqref{PDE}.
Then for any $k\in \RR$, $B_{r} \equiv B_{r}(x_0) \Subset \Omega$ and $\phi \in C_c^{\infty}(B_{r})$
with $0\le \phi\le 1$, we have
\begin{equation}\label{caccio}
\begin{aligned}
&\int_{B_{r}}\int_{B_{r}}\left(\frac{|v(x)-v(y)|}{|x-y|^s}\right)^p\min{\{ \phi(x)^{p},\phi(y)^{p} \}}\, \frac{d\mu(x)\,d\mu(y)}{w(B_{x,y})}\\
&\quad\le c\int_{B_{r}}\int_{B_{r}}\left(\frac{|\phi(x)-\phi(y)|}{|x-y|^s} \max \{v(x),v(y)\} \right)^p\,\frac{d\mu (x)\, d\mu(y)}{w(B_{x,y})}\\
&\qquad +c\bigg(\int_{B_{r}}v\phi^p\,d\mu\bigg)\bigg( \sup_{y\, \in\, \mathrm{supp}\,\phi}\int_{\RRn \setminus B_{r}}\frac{v(x)^{p-1}}{|x-y|^{sp}}\,\frac{d\mu(x)}{w(B_{x,y})}\bigg),
\end{aligned}
\end{equation}
where $v\coloneqq(u-k)_{+}$ and $c>0$ depends on $\lambda$, $\Lambda$ and $p$.

Moreover, if $u\in \Wspw$ is a weak supersolution  to \eqref{PDE}, then the same estimate holds for $v=(u-k)_-$.
\end{proposition}

Proposition \ref{lem.caccio} can be shown almost exactly as in the nondegenerate case, i.e., $w\equiv 1$, where the proof can be found in \cite[Theorem 1.4]{DiCastroKuusiPalatucci16}. The proof utilizes the test function $\phi^p (u-k)_+$, which is in $\Wspw$ due to \Poincare 's inequality Theorem \ref{thm:poincare}.

Later on we need the following iteration lemma, see e.g. \cite[Lemma 7.1]{Giusti-book}.
\begin{lemma}\label{lem:iteration}
Let $(a_j)_{j\geq0}$ be a sequence of nonnegative real numbers satisfying $a_{j+1} \le b_1b_2^{j}a_j^{1+\beta}$ for all $j\ge0$, for some $b_1,\beta>0$ and $b_2>1$. If $a_0\le b_1^{-1/\beta}b_2^{-1/\beta^2}$, then $a_j\to 0$ as $j\to \infty$.
\end{lemma}

We are now ready to prove the local boundedness result Theorem \ref{thm:bounded}.

\begin{proof}[Proof of Theorem~\ref{thm:bounded}]
Suppose that $u$ is a weak subsolution. Fix $B_{r}=B_r(x_0) \Subset \Omega$ and let $k>0$ be a constant to be determined in the last part of the proof. For any $j \in \setN$, write
\begin{align*}
&r_{j}=(1+2^{-j})\frac{r}{2},\ \ \tilde{r}_{j}=\frac{r_{j}+r_{j+1}}{2},\ \ B_{j}=B_{r_{j}}(x_0), \ \ \tilde{B}_{j}=B_{\tilde{r}_{j}}(x_0),\\
&k_{j}=(1-2^{-j})k,\ \ \tilde{k}_{j}=\frac{k_{j}+k_{j+1}}{2},\ \ v_{j}=(u-k_{j})_{+} \ \ \mathrm{and}\ \ \tilde{v}_{j}=(u-\tilde{k}_{j})_{+}.
\end{align*}
Note from the above setting that
\begin{align}\label{lb.rel}
B_{j+1} \subset \tilde{B}_{j} \subset B_{j},\quad k_{j} \le \tilde{k}_{j} \le k_{j+1} \quad \mathrm{and} \quad v_{j+1} \le \tilde{v}_{j} \le v_{j}.
\end{align}
We take any cut-off functions $\phi_{j} \in C_c^{\infty}(\tilde{B}_{j})$ such that $0 \le \phi_{j} \le 1$, $\phi_{j} \equiv 1$ in $B_{j+1}$
and $|\nabla \phi_{j}| \lesssim 2^j/r$. Putting $\phi_{j}$ into the Caccioppoli inequality \eqref{caccio} with $v=\tilde{v}_j$ and dividing the inequality by $w(B_{j+1})$, we get
\begin{equation}\label{lb.split}
\begin{aligned}
\fint_{B_{j+1}}\int_{B_{j+1}}&\left(\frac{|\tilde{v}_{j}(x)-\tilde{v}_{j}(y)|}{|x-y|^s}\right)^p\frac{d\mu(x)\,d\mu(y)}{w(B_{x,y})}\\
&\le c\fint_{B_{j}}\int_{B_{j}}\left(\frac{|\phi_{j}(x)-\phi_{j}(y)|}{|x-y|^s} \max \{ \tilde{v}_{j}(x), \tilde{v}_{j}(y)\} \right)^p\frac{d\mu(x)\,d\mu(y)}{w(B_{x,y})}\\
&\quad +c\bigg(\fint_{B_{j}}\tilde{v}_{j}\phi_{j}^{p}\,d\mu\bigg) \bigg( \sup_{y \,\in\, \mathrm{supp}\,\phi_{j}}\int_{\RRn \setminus B_{j}}\frac{\tilde{v}_{j}(x)^{p-1}}{|x-y|^{sp}}\,\frac{d\mu(x)}{w(B_{x,y})} \bigg)\\
&\eqqcolon c I_1+c(I_2)(I_3).
\end{aligned}
\end{equation}
We first look at $I_1$.
Since $|\phi_{j}(x)-\phi_{j}(y)|\le \|\nabla\phi_{j}\|_{L^\infty} |x-y|  \le c2^{j}|x-y|/r$, and with Lemma~\ref{lem:intweight}, we find
\begin{equation}
\label{lb.rom1}
\begin{split}
I_1
&\le c\,2^{pj}\fint_{B_{j}}\int_{B_{j}} \max \{ \tilde{v}_{j}(x), \tilde{v}_{j}(y)\}^p \frac{ |x-y|^{(1-s)p}}{r^p} \,\frac{d\mu(x)\,d\mu(y)}{w(B_{x,y})}\\
&\lesssim 2^{pj}\fint_{B_{j}}\tilde{v}_{j}(x) ^p\bigg[ \int_{B_{2r}(x)}\frac{|x-y|^{(1-s)p}}{r^p}\,\frac{d\mu(y)}{w(B_{x,y})} \bigg]\, d\mu(x)
\\
&\lesssim \frac{2^{pj}}{1-s}\fint_{B_{j}} \Big(\frac{\tilde{v}_{j}}{r^s}\Big)^p\,d\mu.
\end{split}
\end{equation}
For $I_2$, since $v_{j} \ge \tilde{k}_{j}-k_{j}$ in $\{ u \ge \tilde{k}_{j} \}$, we have that
$ v_{j}^p \ge \tilde{v}_{j}(\tilde{k}_{j}-k_j)^{p-1}$ and hence
\begin{equation}\label{loc.rom2.c1}
I_2 \le \frac{1}{(\tilde{k}_j-k_j)^{p-1}}\fint_{B_{j}} v_{j}^{p}\, d\mu = \frac{2^{(p-1)(j+2)}}{ k^{p-1}}\fint_{B_{j}} v_{j}^{p}\, d\mu.
\end{equation}
In order to estimate $I_3$, we notice that for $x \in \RRn \setminus B_{j}$ and $y \in \tilde{B}_{j}$,
$B_{x,y} \cup B_{x,x_0}  \subset B_{3|x-x_0|}(x_0)$,
$\frac{|x-x_0|}{|x-y|} \le \frac{|x-y|+|y-x_0|}{|x-y|} \le 1+\frac{\tilde{r}_{j}}{r_{j}-\tilde{r}_{j}} \le 2^{j+4}$.
Hence by \eqref{Apinequality},
$$
\frac{w(B_{3|x-x_0|}(x_0))}{w(B_{x,y})} \le c \left( \frac{|x-x_0|}{|x-y|}\right)^{np}\le c\, 2^{npj}
\quad\text{and}\quad
 \frac{w(B_{x,x_0})}{w(B_{3|x-x_0|}(x_0))} \le  c,
$$
which implies that
$$
\frac{1}{w(B_{x,y})} \le  c ~2^{npj} \frac{1}{w(B_{x,x_0})}.
$$
This and \eqref{lb.rel} imply
\begin{align}\label{loc.rom2.c2}
\begin{aligned}
  I_3 &\leq c\,2^{(n+s)pj}\int_{\RRn \setminus B_{r/2}}\frac{v_0(x)^{p-1}}{|x-x_0|^{sp}}\frac{d\mu(x)}{w(B_{x,x_0})}
  \\
  &\leq\frac{c\,2^{(n+1)pj}}{(1-s)r^{sp}}\mathrm{Tail}(u_{+};x_0,r/2)^{p-1}.
\end{aligned}
\end{align}
Let now $\kappa \coloneqq \frac{n}{n-s}\in (1,2)$. Combining \eqref{lb.split}, \eqref{lb.rom1}, \eqref{loc.rom2.c1} and \eqref{loc.rom2.c2}, using the Sobolev-Poincar\'{e} inequality in Theorem~\ref{thm:sobolev_poincare},  and Hölder's inequality, we have
\begin{align}\label{lb.in1}
\begin{aligned}
&\bigg( \fint_{B_{j+1}}\tilde{v}_{j}^{p\kappa}\,d \mu\bigg)^{\frac{1}{\kappa}}
 \lesssim \bigg( \fint_{B_{j+1}}|\tilde{v}_{j}-\mean{\tilde{v}_j}_{w,B_{j+1}}|^{p\kappa}d\mu \bigg)^{\frac{1}{\kappa}}+\bigg( \fint_{B_{j+1}}\tilde{v}_{j}\,d \mu\bigg)^{p} \\
&\qquad \lesssim (1-s)r^{sp}\!\! \fint_{B_{j+1}}\int_{B_{j+1}}\frac{|\tilde{v}_{j}(x)-\tilde{v}_{j}(y)|^p}{|x-y|^{sp}}\frac{d\mu(x)\,d\mu(y)}{w(B_{x,y})}+ \fint_{B_{j+1}}\tilde{v}_{j}^p\,d \mu\\
&\qquad\lesssim 2^{2npj} \left(1 + \frac{\mathrm{Tail}(u_{+};x_0,r/2)^{p-1}}{k^{p-1}}
  \right) \fint_{B_{j}} v_{j}^p\,d\mu.
\end{aligned}
\end{align}
Note that
\begin{align*}
  v_{j+1}(x)>0\implies \tilde{v}_j(x)>k_{j+1}-\tilde{k}_j=2^{-j-2}k.
\end{align*}
We thus have
\begin{align}\label{lb.in3}
  \begin{aligned}
    \bigg( \fint_{B_{j+1}}v_{j+1}^p\, d\mu \bigg)^{\frac{1}{\kappa}}&\leq \bigg( \fint_{B_{j+1}}\tilde{v}_{j}^p \left(\tilde{v}_{j}\frac{2^{j+2}}{k}\right)^{p(\kappa -1)}\, d\mu \bigg)^{\frac{1}{\kappa}}
    \\
    &=k^{-p\frac{\kappa-1}{\kappa}}2^{(j+2)p\frac{\kappa-1}{\kappa}}\bigg( \fint_{B_{j+1}}\tilde{v}_{j}^{p\kappa }\, d\mu \bigg)^{\frac{1}{\kappa}}.
  \end{aligned}
\end{align}
We now define the quantity we want to iterate by
\begin{align*}
  a_j\coloneqq k^{-p}\fint_{B_j} v_j^p\,d\mu.
\end{align*}
With this, our findings so far can be summarized by combining \eqref{lb.in1} and \eqref{lb.in3} to get
\begin{align*}
  a_{j+1}\leq c\,2^{4npj} \left(1+\frac{\mathrm{Tail}(u_+;x_0,r/2)^{p-1}}{k^{p-1}}\right)^\kappa a_j^\kappa .
\end{align*}
If for some $\delta >0$ we have $k \ge  \delta  \,\mathrm{Tail}(u_{+};x_0,r/2))$, then this turns into
\begin{align}\label{eq:iterate}
  a_{j+1}\leq c_1\delta^{-\kappa (p-1)}2^{4npj}  a_j^\kappa ,
\end{align}
for some $c_1>0$ depending on $n,p,\lambda,\Lambda$ and $[w]_{A_p}$. Therefore, by Lemma~\ref{lem:iteration}, if
$$
a_0  = \frac{1}{k^p}\fint_{B_{r}}u_+^p\,d\mu
\le c_{1}^{-\frac{1}{\kappa-1}} \delta^{(p-1)\frac{\kappa}{\kappa-1}}2^{\frac{-4np}{(\kappa -1)^2}},
$$
i.e.,
$$
k^p \ge c_{1}^{\frac{1}{\kappa-1}} \delta^{-(p-1)\frac{\kappa}{\kappa-1}}2^{\frac{4np}{(\kappa -1)^2}} \fint_{B_{r}}u_+^p\,d\mu,
$$
then  we obtain $\lim_{j\to \infty}a_j=0$, which implies that
$$
u \le  k \quad \text{a.e. in } \ B_{r/2}.
$$
Finally, choosing
$$
k \coloneqq  c_{1}^{\frac{1}{p(\kappa-1)}} \delta^{-\frac{p-1}{p}\frac{\kappa}{\kappa-1}}2^{\frac{4n}{(\kappa -1)^2}} \bigg(\fint_{B_{r}}u_+^p\,d\mu \bigg)^{\frac{1}{p}} + \delta \, \mathrm{Tail}(u_{+};x_0,r/2)),
$$
we have the estimate \eqref{lb}, which also implies \eqref{lb1}.
\end{proof}

\subsection{Hölder continuity}

We start with a logarithmic estimate. This will be used in the density improvement result in Lemma~\ref{lem:density}.

\begin{proposition}[Logarithmic estimate] \label{lem.log}
Let $u\in W^{s,p}_w (\RRn)$ be a weak supersolution  to \eqref{PDE} with
$u \ge 0$ in $B_R\equiv B_{R}(x_0) \subset \Omega$. Then for any $d>0$ and $0<r\le \frac{R}{2}$, we have
\begin{equation}\label{lo.in}\begin{split}
r^{sp}\fint_{B_{r}}\int_{B_{r}}\frac{|\log{(u(x)+d)}-\log{(u(y)+d)}|^p}{|x-y|^{sp}}&\, \frac{d\mu(x)\, d\mu(y)}{w(B_{x,y})} \\
&\hspace{-3cm}\le c\left[\frac{1}{s(1-s)} + d^{1-p}\left(\frac{r}{R}\right)^{sp}\mathrm{Tail}(u_{-};x_0,R)^{p-1}\right]
\end{split}\end{equation}
for some $c=c(n,p,\lambda,\Lambda,[w]_{A_p})>0$.
\end{proposition}

\begin{proof}
Write $v(x)\coloneqq u(x)+d$ and fix a cut-off function $\ds \phi \in C^{\infty}_{c}(B_{3r/2})$ such that $0 \le \phi \le 1$, $|\nabla\phi| \le 4/r$ and $\phi \equiv 1$ in $B_{r}$.
Since $\eta\coloneqq v^{-(p-1)} \phi^p $ is nonnegative in $\Omega$ and  belongs to $W^{s,p}_{w,0} (\Omega)$,
we can take it as a test function to find
\begin{align}\label{lo.sp}
\begin{split}
0&\le \int_{B_{2r}}\int_{B_{2r}}k(x,y)\frac{|v(x)-v(y)|^{p-2}(v(x)-v(y))(\eta(x)-\eta(y))}{|x-y|^{sp}} \, dx \, dy
\\
&\quad + 2\int_{\RRn \setminus B_{2r}}\int_{B_{2r}}k(x,y)\frac{|u(x)-u(y)|^{p-2}(u(x)-u(y))\eta(x)}{|x-y|^{sp}}\,dx\,dy
\\
&\eqqcolon I_1+2I_2.
\end{split}
\end{align}
We first consider  $I_1$. Set for $x,y\in B_{2r}$,
\begin{align*}
F(x,y)\coloneqq |v(x)-v(y)|^{p-2} (v(x)-v(y))(v(x)^{-(p-1)} \phi(x)^p -v(y)^{-(p-1)} \phi(y)^p ).
\end{align*}
We start by establishing a pointwise estimate for $F(x,y)$. For this we distinguish the following two cases: $v(y)\le v(x)\le 2 v(y)$ (Case 1) and
$2 v(y) <  v(x)$ (Case 2). By symmetry, we have $F(x,y)=F(y,x)$ and it is thus sufficient to consider these two cases.

\noindent {\bf Case 1 $v(y)\le v(x)\le 2 v(y)$:} We first assume $\phi(x) \ge \phi(y)$. Using the mean value theorem and Young's inequality, we get
\begin{align}\label{lo.rom1.sp}
\begin{split}
F(x,y)&=   (v(x)-v(y))^{p-1} \left(v(x)^{-(p-1)}- v(y)^{-(p-1)}\right)\phi(x)^p\\
&\qquad + (v(x)-v(y))^{p-1} v(y)^{-(p-1)} (\phi(x)^{p}-\phi(y)^{p})\\
&\le  -(p-1) (v(x)-v(y))^{p}  (2v(y))^{-p}\phi(x)^p\\
&\qquad +  p(v(x)-v(y))^{p-1} v(y)^{-(p-1)} \phi(x)^{p-1}(\phi(x)-\phi(y))\\
&\le  -\frac{p-1}{2^{p+1}} \left(\frac{v(x)-v(y)}{v(y)}\right)^{p}\phi(x)^p + c_p (\phi(x)-\phi(y))^p,
\end{split}
\end{align}
for some $c_p>0$ depending on $p$.
We next assume $\phi(x) < \phi(y)$.
Again using the mean value theorem,
\[\begin{split}
F (x,y)&\le  (v(x)-v(y))^{p-1} (v(x)^{-(p-1)} -v(y)^{-(p-1)}) \phi(y)^p\\
&\le -(p-1)\left(\frac{v(x)-v(y)}{v(y)}\right)^p \phi(x)^p.
\end{split}\]
Therefore, since $\log t \le t -1$ and $|\nabla\phi|\le 4/r$, we  have that for every $x,y\in B_{2r}$ with $v(y)\le v(x)\le 2 v(y)$,
\begin{equation}\label{estimateF}
F(x,y) \le - \tilde c [(\log v(x) -\log v(y))\min \lbrace\phi(x),\phi(y)\rbrace ]^p
+c \left( \frac{|x-y|}{r} \right)^{p}
\end{equation}
for some small $\tilde c>0$ and large $c>0$ depending on $p$.

\noindent
\textbf{Case 2 $2 v(y) <  v(x)$:}
We recall the following elementary inequality (see \cite[Lemma 3.1]{DiCastroKuusiPalatucci16}):
For any $p \ge 1$, $\epsilon\in (0,1]$ and $a, b \in \RR^N$,
  \begin{align*}
    |a|^{p} \le (1+c_p\ep)|b|^{p}+(1+c_p\ep)\ep^{1-p}|a-b|^{p},
  \end{align*}
where $c_p > 0$ depends only on $N$ and $p$. Using this inequality together with $2v(y)<v(x)$,
\begin{align*}
F(x,y)&= (v(x)-v(y))^{p-1} \left(v(x)^{-(p-1)}- v(y)^{-(p-1)}\right) \phi(y)^p\\
&\qquad +(v(x)-v(y))^{p-1} v(x)^{-(p-1)} (\phi(x)^{p}-\phi(y)^{p})\\
&\le -\left( 1 -  2^{-(p-1)} \right) (v(x)-v(y))^{p-1}  v(y)^{-(p-1)}  \phi(y)^p\\
&\qquad + c (v(x)-v(y))^{p-1} v(x)^{-(p-1)}( \ep  \phi(y)^{p}  + \ep^{1-p}  |\phi(x)-\phi(y)|^{p})\\
&\le \left(- 1 +  2^{-(p-1)}+c \epsilon \right) \left(\frac{v(x)-v(y)}{v(y)}\right)^{p-1}   \phi(y)^p+  c \ep^{1-p}  |\phi(x)-\phi(y)|^{p}.
\end{align*}
Choosing $\ep>0$ sufficiently small, we have
\begin{align*}
F(x,y) \leq -\tilde c \left(\frac{v(x)-v(y)}{v(y)}\right)^{p-1}\phi(y)^p +c |\phi(x)-\phi(y)|^p.
\end{align*}
Therefore, since
$$
\log v(x) -\log v(y) \le \log \left(\frac{2(v(x)-v(y))}{v(y)}\right) \le \frac{p}{p-1} \left(\frac{2(v(x)-v(y))}{v(y)} \right)^{(p-1)/p},
$$
where we have used the facts that $2v(y) < v(x) $ and $\log t <  \frac{p}{p-1}t^{(p-1)/p}$ for  $t>0$,  and since $|\nabla\phi|\le 4/r$, we also obtain \eqref{estimateF} for $x,y\in B_{2r}$ with $2v(y) < v(x)$. Thus the pointwise estimate \eqref{estimateF} holds for all $x,y\in B_{2r}$.

We are now ready to estimate $I_1$ in \eqref{lo.sp}. Using \eqref{estimateF} and the facts that $\phi\equiv 1$ in $B_r$ and $0\le \phi \le 1$ in $B_{2r}$, and applying \eqref{eq:intweight1} and \eqref{eq:doubling}, we have
\begin{align}\label{estimateI}\begin{split}
I_1&\le -\tilde c \int_{B_{r}}\int_{B_{r}}\frac{|\log{v(x)}-\log{v(y)}|^p}{|x-y|^{sp}}\, \frac{d\mu(x)\,d\mu(y)}{w(B_{x,y})}  \\
&\qquad   + c  r^{-p}\int_{B_{2r}}\int_{B_{4r}(x)} |x-y|^{(1-s)p} \frac{d\mu (y)\, d\mu(x)}{w(B_{x,y})}\\
&\le - \tilde c  \int_{B_{r}}\int_{B_{r}}\frac{|\log{v(x)}-\log{v(y)}|^p}{|x-y|^{sp}}\,\frac{d\mu(x)\,d\mu(y)}{w(B_{x,y})} +c \frac{w(B_{r})}{(1-s)r^{sp}}\\
&\le - \tilde c  \int_{B_{r}}\int_{B_{r}}\frac{|\log{v(x)}-\log{v(y)}|^p}{|x-y|^{sp}}\,\frac{d\mu(x)\,d\mu(y)}{w(B_{x,y})} + c \frac{w(B_{r})}{(1-s)r^{sp}} ,
\end{split}
\end{align}
where $\tilde c>0$ depends on $p$ and $\lambda$ and $c>0$ depends on $n$, $p$, $\Lambda$ and $[w]_{A_p}$.

We next estimate $I_2$. Observe that for $x \in B_{R}$ and $y \in \RRn$,
\begin{align*}
|u(x)-u(y)|^{p-2}(u(x)-u(y))\le \left((u(x)-u(y))_{+}\right)^{p-1}\le c (u(x)^{p-1}+u_{-}(y)^{p-1}).
\end{align*}
Using this and the fact that $\mathrm{supp}\, \phi \subset B_{3r/2}$, we have
\begin{align*}
\begin{split}
I_2 &\le c\int_{B_R \setminus B_{2r}}\int_{B_{3r/2}}\frac{(u(x)-u(y))_{+}^{p-1}v(x)^{-(p-1)}}{|x-y|^{sp}}\frac{d\mu(x)\,d\mu(y)}{w(B_{x,y})}\\
&\qquad + c\int_{\RRn \setminus B_{R}}\int_{B_{3r/2}}\frac{(u(x)^{p-1}+ u_-(y)^{p-1})v(x)^{1-p}}{|x-y|^{sp}}\frac{d\mu(x)\,d\mu(y)}{w(B_{x,y})}\\
&\le c\int_{\RRn \setminus B_{2r}}\int_{B_{3r/2}}\frac{1}{|x-y|^{sp}}\frac{d\mu(x)\,d\mu(y)}{w(B_{x,y})}
 + c d^{1-p}\int_{\RRn \setminus B_{R}}\int_{B_{3r/2}}\frac{u_-(y)^{p-1}}{|x-y|^{sp}}\frac{d\mu(x)\,d\mu(y)}{w(B_{x,y})}\\
&\eqqcolon  c I_{2,1}+c d^{1-p} I_{2,2},
\end{split}
\end{align*}
where in the last estimate we used the facts that  $u \ge 0$ and $v=u+d\ge d$ in $B_{R}$, which implies
$(u(x)-u(y))_{+} \le u(x) \le v(x)$ for $x,y\in B_R$.
By \eqref{eq:intweight2} and \eqref{eq:doubling} we have
\begin{align*}
I_{2,1}  \le c \int_{B_{3r/2}}\bigg(\int_{\RRn \setminus B_{r/2}(x)} \frac{1}{|x-y|^{sp}}\,\frac{d\mu(y)}{w(B_{x,y})}\bigg)\, d\mu(x) \le \frac{c}{sr^{sp}}  w(B_{3r/2}) \le \frac{c}{sr^{sp}}  w(B_r).
\end{align*}
Next we observe that for any $x \in B_{3r/2}=B_{3r/2}(x_0)$ and $y \in \RRn \setminus B_{2r}(x_0)$,
$$
B_{x,y} \cup B_{y,x_0}  \subset B_{3|y-x_0|}(x_0)
$$
and
$$
\frac{|y-x_{0}|}{|x-y|} \le 1+\frac{|x-x_{0}|}{|x-y|} \le 1+\frac{3r/2}{2r-(3r/2)} = 4.
$$
Hence by \eqref{Apinequality}
$$
\frac{w(B_{3|y-x_0|}(x_0))}{w(B_{x,y})} \le c \left( \frac{|y-x_0|}{|x-y|}\right)^{np}\lesssim 1
\quad\text{and}\quad
 \frac{w(B_{y,x_0})}{w(B_{3|y-x_0|}(x_0))} \lesssim  1,
$$
which implies that
$$
\frac{1}{w(B_{x,y})} \le   \frac{c}{w(B_{y,x_0})}.
$$
This enables us to estimate
$$
I_{2,2} \le c w(B_r) R^{-sp}\mathrm{Tail}(u_{-};x_0,R)^{p-1}.
$$
Consequently, we have
\begin{align*}
I_2 \le cw(B_r)r^{-sp} \left[\frac{1}{s} + d^{1-p}\left(\frac{r}{R}\right)^{sp}\mathrm{Tail}(u_{-};x_0,R)^{p-1}\right].
\end{align*}
Inserting this estimate and \eqref{estimateI} into \eqref{lo.sp}, we get \eqref{lo.in}.
\end{proof}

We next show a density improvement result.

\begin{lemma}\label{lem:density}
Let $0<s_0\le s<1$, $\nu >0$, $u\in W^{s,p}_w(\RRn)$ be a weak supersolution to \eqref{PDE} which is nonnegative in $B_R=B_R(x_0)\Subset \Omega$. Let $r\le R\le 4r$. Then for every $\sigma\in(0,1]$, there exists a constant $\epsilon\in (0,\tfrac 12)$ depending only on $n,p,\lambda,\Lambda,s_0,[w]_{A_p}$ and $\sigma$ such that if
\begin{equation}\label{densitycondition}
\frac{w(\{x\in B_{2r} : u(x) \ge \nu \})}{w(B_{2r})} \ge \sigma
\end{equation}
 and
\begin{equation}\label{d.def}
d \coloneqq \left(\frac{r}{R}\right)^{\frac{sp}{p-1}} \mathrm{Tail}(u_{-};x_0,R) \le \epsilon \nu,
\end{equation}
then
\begin{equation}\label{positivity}
\inf_{B_r} u \ge \epsilon \nu.
\end{equation}
\end{lemma}

\begin{proof}
\textit{Step 1.}
We first prove
\begin{equation}\label{density2}
\frac{w(\{x \in B_{2r} : u(x) \le 2\epsilon\nu \})}{w(B_{2r})} \le \frac{\bar{c}}{\sigma\log(1/3\epsilon)}
\end{equation}
holds for every $\epsilon \in (0,\tfrac 14)$, where $\bar{c}\equiv\bar{c}(n,p,\lambda,\Lambda,s_0,[w]_{A_p})>0$. For $x\in B_{2r}$ set
\begin{equation*}
h(x) \coloneqq \min\left\{\left[\log\left(\frac{\nu+d}{u(x)+d}\right)\right]_{+},\,\log\left(\frac{1}{3\epsilon}\right)\right\}.
\end{equation*}
Note that $h(x)=0$ if and only if $u(x)\geq \nu$. Also note that
\begin{align*}
	\abs{h(x)-h(y)}&\leq \biggabs{\log \left(\frac{\nu+d}{u(x)+d}\right) -\log\left(\frac{\nu+d}{u(y)+d}\right) }
	\\
	&=\abs{\log(u(x)+d)-\log(u(y)+d)}.
\end{align*}
By the \Poincare~inequality Theorem~\ref{thm:poincare} and Proposition~\ref{lem.log} with $d$ as in \eqref{d.def} we have
\begin{equation}\label{vv}
\begin{aligned}
&\fint_{B_{2r}}|h-\mean{h}_{w,B_{2r}}|^p\,d\mu \\
&\le c(1-s) r^{sp}\fint_{B_{2r}}\int_{B_{2r}}\frac{|\log(u(x)+d)-\log(u(y)+d)|^p}{|x-y|^{sp}}\,\frac{d\mu(x)\,d\mu(y)}{w(B_{x,y})}
\\
&\le c.
\end{aligned}
\end{equation}
We note from assumption \eqref{densitycondition} and $0\le h\le \log(1/3\epsilon)$, that
$$
\begin{aligned}
\log\left(
 \frac{1}{3\epsilon}\right)  & = \frac{1}{w(\{x\in B_{2r}: h(x)=0\})}\int_{\{x\in B_{2r}: h(x)=0\}}\left[\log\left(
 \frac{1}{3\epsilon}\right)  - h\right]\,d\mu \\
&\le \frac{1}{\sigma}\left[\log\left(
 \frac{1}{3\epsilon}\right) -\mean{h}_{w,B_{2r}}\right].
\end{aligned}
$$
Integrating both sides with the measure $d\mu$ over the set $\left\{x\in B_{2r} : h(x)=\log(1/3\epsilon)\right\}$ and using \eqref{vv}, we find that
\begin{align*}
\frac{w\left(\left\{x\in B_{2r} : h(x)=\log(1/3\epsilon)\right\}\right)}{w(B_{2r})} &\le \frac{1}{\sigma \log(1/3\epsilon)}\fint_{B_{2r}}|h-\mean{h}_{w,B_{2r}}|\,d\mu 
\\
&\le \frac{\bar{c}}{\sigma \log(1/3\epsilon)}.
\end{align*}
By the definitions of $h$ we have the following implications
\begin{equation*}
u(x) \le 2\epsilon \nu
\ \ \Longrightarrow\ \ u(x)+d \le 3\epsilon(\nu+d)
 \ \ \Longrightarrow\ \ h(x)=\log\left(
 \frac{1}{3\epsilon}\right).
\end{equation*}
hence we get \eqref{density2} from the above inequality.

\textit{Step 2.} Now, we prove \eqref{positivity} by applying the De Giorgi iteration and choosing $\epsilon>0$ small enough. For each $j \in \mathbb{N}$, we set
\begin{equation*}
\rho_{j} \coloneqq (1 +2^{-j})r, \qquad \tilde{\rho}_{j} \coloneqq \frac{\rho_{j}+\rho_{j+1}}{2}, \qquad B_{j} \coloneqq B_{\rho_{j}}(x_{0})
\end{equation*}
and then choose cut-off functions $\phi_{j} \in C^{\infty}_{0}(B_{\tilde{\rho}_{j}})$ such that
$0 \le \phi_{j} \le 1$, $\phi_{j} \equiv 1$ in $B_{j+1}$ and  $|\nabla\phi_{j}| \le 2^{j+1}/r$.
We further define
\begin{equation*}
\ell_{j} \coloneqq (1+2^{-j})\epsilon \nu,
 \quad v_{j} \coloneqq (\ell_{j}-u)_{+},
 \qquad A_{j} \coloneqq \frac{w(\{x\in B_{j}:u(x)<\ell_{j}\})}{w(B_{j})}.
\end{equation*}
Note that
\begin{equation}\label{parameter.range}
\rho_{j},\tilde{\rho}_{j} \in (r,2r),
\quad d \le \epsilon \nu \le \ell_{j} \le 2\epsilon \nu,
\quad \ell_{j} - \ell_{j+1} = 2^{-j-1}\epsilon \nu \ge 2^{-j-2}\ell_{j}.
\end{equation}
Furthermore, \eqref{density2} implies
\begin{equation}\label{A0}
A_{0} = \frac{w(\{x\in B_{0}:u(x)<2\epsilon\nu\})}{w(B_{0})} \le \frac{\bar{c}}{\sigma\log(1/3\epsilon)} .
\end{equation}
Set $\kappa = \tfrac{n}{n-s_0}$. Using the definitions of $A_j$ and $\ell_j$ with the fact that $v_{j} \ge (\ell_{j}-\ell_{j+1})\indicatorset{u<\ell_{j+1}}$, the Sobolev-Poincar\'e inequality in Theorem~\ref{thm:sobolev_poincare}, H\"older's inequality and \eqref{eq:doubling}, we have
\begin{equation}\label{lhs.j+1}\begin{aligned}
A_{j+1}^{1/\kappa} (\ell_{j}\!-\!\ell_{j+1})^p
&= \bigg( \frac{1}{w(B_{j+1})}\int_{\{x\in B_{j+1} : u(x) < \ell_{j+1}\}}[\ell_{j}\!-\!\ell_{j+1}]^{p\kappa}\,d\mu \bigg)^{\frac{1}{\kappa}}  
\\
&\le \bigg(\fint_{B_{j+1}}v_{j}^{p\kappa}\,d\mu\bigg)^{\frac{1}{\kappa}}  \\
& \lesssim (1-s)r^{sp}\fint_{B_{j+1}}\int_{B_{j+1}}\!\!\frac{|v_{j}(x)-v_{j}(y)|^p}{|x-y|^{sp}}\,\frac{d\mu (x)\, d\mu(y)}{w(B_{x,y})}
+ \!\!\fint_{B_{j+1}}v_{j}^p\,d\mu.
\end{aligned}\end{equation}
We now estimate each term on the right-hand side of \eqref{lhs.j+1}. First we have that
\begin{equation}\label{Gwj}
\fint_{B_{j+1}}v_{j}^p\,d\mu \le c \fint_{B_{j}}v_{j}^p \,d\mu \le \frac{1}{w(B_{j})}\int_{\{x\in B_{j}:  u(x)\le\ell_{j}\}}\ell_{j}^{p}\,d\mu \le \ell_{j}^pA_{j}
\end{equation}
Next, we apply the \Caccioppoli~estimate from Proposition \ref{lem.caccio} to $v_{j}$, $B_{j}$ and $\phi_j$ and the same approach used in the estimate of \eqref{lb.in1}, to obtain
\begin{equation}\label{Caccio.jth}
\begin{aligned}
&\fint_{B_{j+1}}\int_{B_{j+1}}\frac{|v_{j}(x)-v_{j}(y)|^p}{|x-y|^{sp}}\,\frac{d\mu(x)\,d\mu(y)}{w(B_{x,y})} \\
&\le c\fint_{B_{j}}\int_{B_{j}}\left(\frac{|\phi_j(x)-\phi_j(y)|}{|x-y|^s} \max \{v_j(x),v_j(y)\} \right)^p\frac{d\mu(x)\,d\mu(y)}{w(B_{x,y})}\\
&\qquad +c\bigg(\fint_{B_{j}}v_{j}\phi_j^p\,d\mu\bigg)\bigg( \sup_{y\in B_{\tilde\rho_j}}\int_{\RRn \setminus B_{j}}\frac{v_{j}(x)^{p-1}}{|x-y|^{sp}}\frac{d\mu(x)}{w(B_{x,y})}\bigg)\\
& \le \frac{c2^{jp}}{(1-s) r^{sp}}\ell_{j}^{p}A_j +  c 2^{j(n+sp)}\ell_{j}A_j\int_{\mathbb{R}^{n}\setminus B_{j}}\frac{(\ell_{j}+u_{-}(x))^{p-1}}{|x-x_{0}|^{sp}} \frac{d\mu (x)}{w(B_{x,x_0})}\\
& \le  \frac{c2^{j(n+sp)}}{(1-s) r^{sp}}  \left(\ell_{j}^{p}   + \ell_{j}d^{p-1} \right) A_j \le \frac{c 2^{j(n+p)}}{(1-s) r^{sp}} \ell_{j}^{p}A_j.
\end{aligned}
\end{equation}
Combining all the estimates \eqref{lhs.j+1}, \eqref{Gwj} and  \eqref{Caccio.jth}, and using $\eqref{parameter.range}_3$, we obtain
\begin{equation*}
A_{j+1}^{1/\kappa} \le c_{2}2^{j(n+2p)}A_{j}
\end{equation*}
for a constant $c_{2}\equiv c_{2}(n,p,\lambda, \Lambda, s_0,[w]_{A_p})>0$.
Therefore, if choose $\epsilon$ small so that
\begin{equation}\label{delta.choice}
A_0\le \frac{\bar{c}}{\sigma\log(1/3\epsilon)}  \le c_2^{-\frac{\kappa}{\kappa-1}}2^{-\frac{(n+2p)\kappa}{(\kappa-1)^2}},
\end{equation}
by Lemma~\ref{lem:iteration}, we obtain $A_{j} \rightarrow 0$. This implies the desired inequality \eqref{positivity}.
\end{proof}

Now, we prove the H\"older continuity of the weak solution to \eqref{PDE}.

\begin{proof}[Proof of Theorem~\ref{thm:Holder}]
Let $B_{r} \equiv B_{r}(x_0) \Subset \Omega$. Set
\begin{align}\label{def.no0}
\nu_{0}
\coloneqq 2\bigg\{ c_b\bigg( \fint_{B_r}|u|^p\,d\mu \bigg)^{1/p} +\mathrm{Tail}(u;x_0,r/2)\bigg\},
\end{align}
where $c_{b}>0$ is the constant given in Theorem~\ref{thm:bounded}.
It is sufficient to show that there exists $\alpha\in(0,1)$ and $\tau\in(0,1)$ such that
\begin{equation}\label{eq:oscillationdecay}
\mathrm{osc}_{B_{\tau^jr/2}}  \le \tau^{\alpha j} \nu_0   \quad \text{for every }\ j\in \mathbb{N}_0.
\end{equation}
We prove this by induction. The case $j=0$ follows directly from Theorem~\ref{thm:bounded}.
We suppose now that \eqref{eq:oscillationdecay} holds for all $j\in\{0,1,\dots,i\}$ and aim to show \eqref{eq:oscillationdecay} for $j=i+1$.

For $j\in \mathbb N$ we set
\begin{align}
\label{rj}
r_j\coloneqq\tau^{j}\frac{r}{2},
\quad
B_j=B_{r_j}(x_0),
\quad
\nu_{j}\coloneqq\left( \frac{r_j}{r_0} \right)^{\alpha}\nu_{0}=\tau^{\alpha j}\nu_{0},
\end{align}
as well as
\begin{align*}
  M_j \coloneqq \sup_{B_j}\, u
  \quad\text{and}\quad
  m_j \coloneqq \inf_{B_j}\, u,
\end{align*}
where $\alpha$ and $\tau$ are small positive constants to be determined later. In particular, we assume that $0<\tau\le1/4$ and $0<\alpha \le \frac{s_0p}{2(p-1)}$

Then we have the following two cases:
\begin{align*}
\begin{aligned}
&\text{\textbf{(Case 1)}}:\quad
w(\{x\in 2B_{i+1}: u-m_i \ge \tfrac{1}{2}\nu_i\})\ge \tfrac{1}{2}w(2B_{i+1}),\\
&\text{\textbf{(Case 2)}}:\quad
w(\{x\in 2B_{i+1}: \nu_i-(u-m_i) \ge \tfrac{1}{2}\nu_i\})\ge \tfrac{1}{2}w(2B_{i+1}).
\end{aligned}
\end{align*}
We further define
$$
u_i\coloneqq
\begin{cases}
u-m_i & \text{(Case 1)},\\
\nu_i-(u-m_i)  &\text{(Case 2)},
\end{cases}
\quad\text{and}\quad
d_i \coloneqq\tau^{\frac{sp}{p-1}}\mathrm{Tail}((u_i)_{-};x_0,r_i).
$$
Note that, in both cases, $u_i$ is also a weak solution to \eqref{PDE} and nonnegative in $B_i$, and satisfies that $|u_i| \le M_j-m_j + \nu_i \le \nu_j+\nu_i\le 2\nu_j$ in $B_j$ for $j\in\{0,1,\dots,i\}$ and $ |u_i|\le |u|+2\nu_0$ in $\RRn\setminus B_0$.

Using these estimates, \eqref{eq:intweight2} and \eqref{def.no0}, we have
\begin{equation}\label{ho.tail.sp}
\begin{aligned}
&\frac{1}{(1-s)r_i^{sp}}\, \mathrm{Tail}((u_i)_-;x_0,r_i)^{p-1}\\
&=\sum_{j=1}^{i} \int_{B_{j-1} \setminus B_j} \frac{|u_{i}(x)|^{p-1}}{|x-x_0|^{sp}} \,\frac{d\mu(x)}{w(B_{x,x_0})} +\int_{\RRn \setminus B_0} \frac{|u_{i}(x)|^{p-1}}{|x-x_0|^{sp}} \,\frac{d\mu(x)}{w(B_{x,x_0})} \\
&\le  \sum_{j=1}^{i} \int_{B_{j-1} \setminus B_j} \frac{(2\nu_{j-1})^{p-1}}{|x-x_0|^{sp}} \,\frac{d\mu(x)}{w(B_{x,x_0})} +  \int_{\RRn \setminus B_0} \frac{(|u(x)|+2\nu_0)^{p-1}}{|x-x_0|^{sp}} \frac{dw(x)}{w(B_{x,x_0})} \\
&\le c  \sum_{j=1}^{i}  \frac{\nu_{j-1}^{p-1}}{r_j^{sp}} +c  \frac{\nu_0^{p-1}}{r_0^{sp}} + \frac{c}{(1-s)r_0^{sp}}\, \mathrm{Tail}(u_0;x_0,r_0)^{p-1}
\\
&\le \frac{c}{1-s}  \sum_{j=1}^{i}  \frac{\nu_{j-1}^{p-1}}{r_j^{sp}} .
\end{aligned}\end{equation}
Then, from the above inequality and the definitions of $r_i$ and $\nu_i$, we have
\begin{align*}
d^{p-1}
\le c  \nu_{i}^{p-1} \sum_{j=1}^{i}\tau^{(1+i-j)\{sp-\alpha(p-1)\}}
\le c  \nu_{i}^{p-1} \sum_{j=1}^{i}\tau^{jsp_0/2}
\le c\frac{\tau^{s_0p/2}}{1-\tau^{sp_0/2}}\nu_i^{p-1}.
\end{align*}
Here we choose $\tau$ small so that
$$
 d \le \left(c \frac{\tau^{s_0p/2}}{1-\tau^{sp_0/2}}\right)^{1/(p-1)}\nu_i\le  \frac{\epsilon}{2}\nu_i,
$$
where $\epsilon$ is the constant in Lemma~\ref{lem:density} when $\sigma=\tfrac 12$. 

Therefore, in light of Lemma~\ref{lem:density} for $\nu=\nu_i/2$, $\sigma=\tfrac 12$ and $B_r=B_{i+1}$, we have
$m_{i+1}=\inf_{B_{i+1}}\,u_{i+1} \ge \epsilon \nu_i/2$, which implies that
$$\begin{array}{ll}
\text{(Case 1)}&
m_{i+1}- m_{i} \le\frac{\epsilon}{2}\nu_i,\\
\text{(Case 2)}&
M_i -M_{i+1} \le \nu_i + m_i- M_{i+1} \le \frac{\epsilon}{2}\nu_i,
\end{array}$$
Moreover, since $M_i-m_i\le \nu_i$, $M_{i+1}\le M_i$ and $m_{i}\le m_{i+1}$, we have in both cases
$$
M_{i+1}-m_{i+1} \le (1-\tfrac{\epsilon}{2}) \nu_i = (1-\tfrac{\epsilon}{2})\tau^{-\alpha}\nu_{i+1}.
$$
Finally, choosing $\alpha$ sufficiently small such that $ (1-\epsilon/2)\le \tau^{\alpha}$, we obtain \eqref{eq:oscillationdecay} when $j=i+1$.
\end{proof}

\subsection{Harnack's inequaltiy}\label{sbusec:Harnack}

We start with recalling the following Krylov-Safonov type~\cite{KylovSafanov1980} covering lemma for a doubling measure $w$. We use the version stated by Kinnunen and Shamugalingam~\cite[Lemma 7.2]{KinnuenShamugalingam01}.

\begin{lemma}\label{lem:covering} Let $w$ be a doubling measure, $E\subset B_r(x_0)$ be a $w$-measurable set, and $\delta\in(0,1)$. Define
$$
E_{\delta}\coloneqq \bigcup_{0<\rho\le 3r/2} \{B_{3\rho}(x)\cap B_r(x_0): x\in B_r(x_0),\ w(E\cap B_{3\rho}(x))\ge \delta w(B_{\rho}(x)) \}.
$$
Then, either $E_{\delta}=B_r(x_0)$, or else $w(E_{\delta})\ge (c_w\delta)^{-1}w(E)$, where $c_w\ge 1$ is the doubling constant of $w$.
\end{lemma}

Using this covering lemma and Lemma~\ref{lem:density}, we obtain the following weak Harnack inequality.

\begin{theorem}\label{thm:weak.harnack} (Weak Harnack inequality)
Let $u\in W^{s,p}_w(\RRn)$ be a weak supersolution to \eqref{PDE}. There exist constants $p_{0}>0$ and $c >0$ depending on $n,p,\lambda,\Lambda,s_0$ and $[w]_{A_p}$ such that if $u$ is nonnegative in $B_{2r} = B_{2r}(x_{0}) \Subset \Omega$, then there holds
$$
\bigg(\fint_{B_{r}}u^{p_{0}}\,d\mu\bigg)^{\frac{1}{p_{0}}}  \le c\inf_{B_{r}}u + c\,\mathrm{Tail}(u_-;x_0,2r).
$$
\end{theorem}
\begin{proof}  Set $\delta\coloneqq\frac{1}{2c_w}$ and $T\coloneqq c_3\mathrm{Tail}(u_-;x_0,2r)$, where the constant $c_3>0$ will be determined in \eqref{def.c3} below.
For each $i \in \mathbb{N}_0$ and $t>0$, we define
\begin{equation*}
A^{i}_{t} \coloneqq \left\{ x\in B_{r}:u(x)>t\epsilon^{i}-\frac{T}{1-\epsilon}\right\},
\end{equation*}
where $\epsilon \in (0,1/4)$ is the constant determined in Lemma~\ref{lem:density} when $\sigma=\frac{\delta}{6^{np}[w]_{A_p}}$.
Obviously, $A^{i-1}_{t} \subset A^{i}_{t}$.
 For $x \in B_{r}$ suppose that
$$
w(A^{i-1}_{t}\cap B_{3\rho}(x) )\ge  \delta w(B_{\rho}(x)).
$$
Then by \eqref{Apinequality} we have
$$
\frac{w(A^{i-1}_{t}\cap B_{6\rho}(x) )}{w(B_{6\rho}(x))}\ge   \frac{1}{6^{np}[w]_{A_p}}\frac{w(A^{i-1}_{t}\cap B_{3\rho}(x) )}{w(B_{\rho}(x))} > \frac{\delta}{6^{np}[w]_{A_p}}.
$$
In light of Lemma~\ref{lem:density}, to  $r=3\rho$, $\nu = t\epsilon^{i-1} - \frac{T}{1-\epsilon}$, $\sigma=\frac{\delta}{3^{np}[w]_{A_p}}$ and $d=2^{sp/(p-1)}\mathrm{Tail}(u_-;x,6\rho)$, we have
$$
\inf_{B_{3\rho}}u \ge \epsilon   \left(t\epsilon^{i-1}-\frac{T}{1-\epsilon}\right)  - d.
$$
Moreover, using the same argument as in the estimation of \eqref{loc.rom2.c2} and the facts that $\rho\le 2r/3$ and $u\ge 0$ in $B_{2r}$, it follows that
\begin{equation}\label{def.c3}
\begin{aligned}
\frac{d^{p-1}}{1-s}&= (3\rho)^{sp} \int_{\mathbb{R}^{n}\setminus B_{6\rho}(x)}\frac{u_{-}^{p-1}(y) }{|y-x|^{sp}}\,\frac{d\mu(y)}{w(B_{y,x})}
\\
&\le  (2r)^{sp} \int_{\RRn\setminus B_{2r}}\frac{u_{-}^{p-1}(y) }{|y-x|^{sp}}\,\frac{d\mu(y)}{w(B_{y,x})}\\
& \le c_{3}^{p-1}(2r)^{sp} \int_{\mathbb{R}^{n}\setminus B_{2r}}\frac{u_{-}^{p-1}(y) }{|y-x_0|^{sp}}\,\frac{d\mu(y)}{w(B_{y,x_0})}
= \frac{T^{p-1}}{1-s}
\end{aligned}
\end{equation}
for some $c_3>0$ depending on $n,p,\lambda,\Lambda,s_0$ and $[w]_{A_p}$. Therefore,
$$
\inf_{B_{3\rho}}u \ge \epsilon   \left(t\epsilon^{i-1}-\frac{T}{1-\epsilon}\right)  - T = t\epsilon^{i}-\frac{T}{1-\epsilon}.
$$
This implies $B_{3\rho}(x)\cap B_r(x_0)\subset A^{i}_{t}$, and hence, by the definition $E_{\delta}$, $(A^{i-1}_{t})_{\delta} \subset A^{i}_{t}$. Finally, by  Lemma~\ref{lem:covering} we have that for each $i\in \mathbb{N}$ and $t>0$, either $A^{i}_{t} = B_{r}$ or else $w(A^{i}_{t}) \ge 2 w(A^{i-1}_{t})$.

We claim that if
\begin{equation}\label{m.condition}
w(A^0_t) >  2^{-m} w(B_r)
\end{equation}
for some $m\in\mathbb N$, then $A^m_t=B_r$ and so
$$
u(x)>t\epsilon^{m}-\frac{T}{1-\epsilon} \quad \text{in }\ B_r.
$$
Indeed, if $A^{i}_{t} = B_{r}$ for some $1\le i \le m-1$ then $ B_r=A^{i}_{t} \subset A^m_t\subset B_r$, and  if $w(A^{i}_{t}) \ge 2 w(A^{i-1}_{t})$ for all $1 \le i \le m-1$ then
$$
w(A^{m-1}_{t})\ge  2 w(A^{m-1}_{t})  \ge \cdots \ge 2^{m-1} w(A^0_t) > 2^{-1} w(B_r),
$$
which implies $A^{m}_{t} = B_{r}$, since if $w(A^{m}_{t}) \ge 2w(A^{m-1}_{t})$ we have the contradictory  inequality $w(A^{m}_{t}) > w(B_r)$.

We next claim that for each $t>0$,
\begin{equation}\label{density.upbound}
\frac{w(\{u>t\}\cup B_r)}{w(B_r)} \le  \frac{w(A^0_t)}{w(B_r)} \le \frac{1}{\epsilon^\beta t^\beta } \left(\inf_{B_r}u +\frac{T}{1-\epsilon}\right)^\beta,
\end{equation}
where $\beta=\log_\epsilon{(1/2)}$. If $w(A^0_t)=0$, then the inequality is trivial. Suppose $w(A^0_t)>0$. Choose $m$ as the smallest positive integer satisfying \eqref{m.condition}, i.e. $$
m >\log_{1/2}{\left(\frac{w(A^0_t)}{w(B_r)}\right)} \ge m-1.
$$
Therefore, we have from the previous claim that
$$
\inf_{B_r} u >t \epsilon\left(\frac{w(A^0_t)}{w(B_r)}\right)^{1/\beta}  -\frac{T}{1-\epsilon},
$$
which yields \eqref{density.upbound}.

Finally, choosing $p_0=\beta/2$ and $a=\inf_{B_r} u +\frac{T}{1-\epsilon}$,
$$\begin{aligned}
\fint_{B_r} u^{p_0} \, d\mu & = p_0 \int_0^\infty t^{p_0-1} \frac{w(\{u>t\}\cup B_r)}{w(B_r)} \,dt \\
&\le  p_0\int_0^a t^{p_0-1} \,dt
+  \frac{p_0}{ \epsilon^{\beta} } \left(\inf_{B_r}u +\frac{T}{1-\epsilon}\right)^{2p_0} \int_a^\infty t^{-1-p_0}   \,dt\\
& = \left(1+ \frac{1}{\epsilon^\beta }\right) \left(\inf_{B_r}u +\frac{T}{1-\epsilon}\right)^{p_0} .
\end{aligned}$$
This completes the proof.
\end{proof}

The following lemma implies that the tail of $u_+$ is controlled by the supremum of $u$ and  the tail of $u_-$. The proof is almost the same as the one in \cite[Lemma 4.2]{DiCastroKuusiPalatucci14} involving the tracing of the constant $s$ and consideration of the $A_p$ weight $w$ as in the estimation of \eqref{lb.split}. Hence we omit the proof.

\begin{lemma}\label{lem:tail.pm}
Let $u \in W^{s,p}_w(\RRn)$ be a weak solution to \eqref{PDE} which is nonnegative and bounded in $B_{R}\equiv B_{R}(x_{0})$. Then for any $0<r\le R$, there holds
$$\begin{aligned}
\mathrm{Tail}(u_+;x_0,r) \le c \sup_{B_{r}}u + c \left(\frac{r}{R}\right)^{sp/(p-1)} \mathrm{Tail} (u_-;x_0,R).
\end{aligned}
$$
holds whenever , where $c\equiv c(n,p,\lambda, \Lambda,s_0,[w]_{A_p})>0$.
\end{lemma}

Finally, we are in a position to prove Theorem \ref{thm:Harnack}.

\begin{proof}[Proof of Theorem~\ref{thm:Harnack}]
Note that from the supremum estimate in Theorem \ref{thm:bounded}, using standard interpolation and covering arguments, we  have
$$\begin{aligned}
\underset{B_{\sigma_1r}}{\sup}\, u_{+} \le c_q  \frac{\delta^{-\frac{p-1}{p}\frac{n}{s_0}}}{(\sigma_2-\sigma_1)^{n/q}}\bigg(\fint_{B_{\sigma_2 r}}u_+^{q}\,dw \bigg)^{\frac{1}{q}} + c\delta\, \mathrm{Tail}(u_{+};x_0,\sigma_2r),
\end{aligned}$$
for every $q>0$, $\tfrac 12 \le \sigma_1 < \sigma_2 \le 1$ and $\delta \in (0,1)$.
This together with Theorem~\ref{thm:weak.harnack} and Lemma~\ref{lem:tail.pm} yields
$$\begin{aligned}
\sup_{B_{\sigma_1r}} u_{+}\!
& \le c  \frac{\delta^{-\frac{p-1}{p}\frac{n}{s_0}}}{(\sigma_2\!-\!\sigma_1)^{n/\epsilon_0}}\bigg(\fint_{B_{\sigma_2 r}}u_+^{p_0}\,d\mu \bigg)^{\frac{1}{\epsilon_0}} + c \delta\, \mathrm{Tail}(u_{+};x_0,\sigma_2r)\\
& \le  \frac{c\delta^{-\frac{p-1}{p}\frac{n}{s_0}}}{(\sigma_2\!-\!\sigma_1)^{n/\epsilon_0}}\left(\inf_{B_{\sigma_2r}}u + c \mathrm{Tail}(u_{-};x_0,2\sigma_2r )\right)
 + c \delta  \sup_{B_{\sigma_2r}}u + c \delta  \mathrm{Tail}(u_{-};x_0,2r ).
\end{aligned}$$
Therefore, choosing $\delta$ sufficiently small, we obtain for any $1/2\le \sigma_1 < \sigma_2 \le 1$
$$
\underset{B_{\sigma_1r}}{\sup}\, u_{+}
 \le   \frac12  \sup_{B_{\sigma_2r}}u^{p-1} + \frac{c}{(\sigma_2-\sigma_1)^{n/\epsilon_0}}\inf_{B_{r}}u + c \mathrm{Tail}(u_{-};x_0,2r).
$$
From here, a standard iteration argument yields the Harnack inequality \eqref{eq:Harnack}.
\end{proof}


\appendix
\section{Riesz potential estimate}\label{sec:appendix}

In this section we prove a nonlocal replacement  of the Riesz type estimate
\begin{align*}
  \abs{v(x) - \mean{v}_B} &\lesssim \int_B \frac{\abs{\nabla v(y)}}{\abs{x-y}^{n-1}}\,dy.
\end{align*}
In particular, we prove the following estimate:
\begin{lemma}\label{lem:Rieszpotential}
  Let $\alpha \in (0,1)$, $B$ be a ball and $\psi\in L^\infty(B)$ be non-negative with $\norm{\psi}_{L^1(B)}=1$ and $\norm{\psi}_{L^\infty(B)}\leq c_0r^{-n}$. Then we have for all $v\in L^1(B)$ and every Lebesgue point $x \in \overline{B}$ of $v$
  \begin{align}
    \label{eq:Rieszpotential}
    \begin{aligned}
      \abs{v(x) - \mean{v}_\psi} &\lesssim (1\!-\!\alpha)\int_B\!\!\int_B \frac{|v(\zeta)\!-\!v(z)|}{\abs{\zeta-z}^{n+\alpha}(\abs{x\!-\!z}+\abs{x\!-\!\zeta})^{n-\alpha}} \, d\zeta\,dz\hspace*{-3mm}
      \\
      &\leq (1\!-\!\alpha)\int_B\!\!\int_B \frac{|v(\zeta)\!-\!v(z)|}{\abs{\zeta-z}^{n+\alpha}}\,d\zeta \frac{1}{\abs{x-z}^{n-\alpha}} \,dz.
    \end{aligned}
  \end{align}
  The hidden constant depends only on $n$ and $c_0$.
\end{lemma}

\begin{proof}
  For fixed $\alpha \in (0,1)$ we define the radial weakly singular weight
  \begin{align*}
    \eta_1(x) &\coloneqq \frac{(1-\alpha)n}{\omega_{n-1}(n-1+\alpha)} \indicator_{B_1(0)} (\abs{x}^{-n-\alpha+1} -1),
  \end{align*}
  where $\omega_{n-1}= 2\pi^{n/2}/\Gamma(n/2)$ is the surface area of the $n$-ball.  Then $\eta_1 \geq 0$, $\eta_1 \in W^{1,1}(\RRd \setminus \set{0})$ and
  \begin{align*}
    \int \!\eta_1(x) \,dx &= \tfrac{(1-\alpha)n}{(n-1+\alpha)}  \int_0^1 r^{-\alpha} -r^{n-1}\,dr = \tfrac{(1-\alpha)n}{(n-1+\alpha)} \Big( \tfrac{1}{1-\alpha} -\tfrac{1}{n}\Big) = 1.
  \end{align*}
  We then define $\eta_r(x) \coloneqq r^{-n} \eta_1(x/r)$.

  Let $B_1=B_1(0)$ and $B_2=B_2(0)$. By scaling and translation it suffices to prove the claim for $B_2$.  We assume that~$x$ is a Lebesgue point of~$v$.  Set $\omega \coloneqq\abs{B_1(0)}^{-1}\indicator_{B_1(0)}\ast \eta_1$. Note that $\omega$ is supported in $B_2$, $w\geq0$ and $\norm{\omega}_1=1$.

  We prove the statement in two steps: First we show the estimate in the case $\psi =\omega$ and then we reduce the case of a general weight $\psi$ to that. Throughout the proof, we make sure that all (implicit) constants depend only on $n$ but not on $\alpha$.

  \textbf{Step 1:} We have that
  \begin{align*}
    \abs{v(x) - \mean{v}_\omega}
    &= \biggabs{\fint_{B_1} \!\big(v(x) - (v* \eta_{1})(y)\big) \,dy}
    \leq \fint_{B_1} \abs{v(x) - (v* \eta_{1})(y)}\,dy.
  \end{align*}
  We split the integral into
  \begin{align*}
    \fint_{B_1} \abs{v(x) - (v* \eta_{1})(y)}\,dy
    &\leq
    \fint_{B_1}\abs{v(x) - (v* \eta_{\frac{\abs{x-y}}{2}})(y)}\,dy
    \\
    &\quad
    +\fint_{B_1}\abs{v(y) - (v* \eta_{\frac{\abs{x-y}}{2}})(y)}\,dy
    \\
    &\quad
    +\fint_{B_1}\abs{v(y) - (v* \eta_{1})(y)}\,dy\eqqcolon \mathrm{I}+\mathrm{II}+\mathrm{III}.
  \end{align*}
  We start by estimating $\mathrm{I}$. Since $x$ is a Lebesgue point of~$v$, we can calculate
  \begin{align*}
    \lefteqn{v(x) - (v*\eta_{\frac{\abs{x-y}}{2}})(y)} \qquad
    &
    \\
    &= -\int_0^1 \frac{d}{dt} \Big((v * \eta_{t\frac{\abs{x-y}}2})\big(x + t(y-x)\big)\Big) \,dt
    \\
    &= -\int_0^1 \int v(z) \frac{d}{dt} \eta_{t\frac{\abs{x-y}}{2}}\big(x+t(y-x)-z\big)  \,dz \,dt.
    \\
    &= \int_0^1  \int \big(v(x+t(y-x)\big) - v(z) \big) \frac{d}{dt}  \eta_{t\frac{\abs{x-y}}{2}}\big(x+t(y-x)-z\big)  \,dz \,dt,
  \end{align*}
  where we used in the last step that $\int \!\frac{d}{dt} \eta_{t\frac{\abs{x-y}}{2}}(x+t(y\!-\!x)-z) \,dz=0$.
  Now,
  \begin{align*}
    \lefteqn{\frac{d}{dt}  \eta_{t \frac{\abs{x-y}}{2}}\big(x+t(y-x)-z\big)} \quad
    &
    \\
    &=
    \frac{d}{dt}  \Bigg( \Big(\frac{\abs{x-y}}{2}t\Big)^{-n} \eta_{1}\bigg(2\,\frac{x+t(y-x)-z}{t \abs{x-y}}\bigg) \Bigg)
    \\
    &= -\frac{d}{t} \eta_{t \frac{\abs{x-y}}{2}}\big(x+t(y\!-\!x)-z\big)
    - \Big( \frac{\abs{x\!-\!y}}{2} t \Big)^{-n} \!\nabla \eta_{1} \bigg(2\frac{x+t(y\!-\!x)-\!z}{t\abs{x\!-\!y}}\bigg) \cdot \frac{2(x\!-\!z)}{\abs{x\!-\!y}t^2}.
  \end{align*}
  Hence, with $[x,y]_t \coloneqq x+t(y-x)$ we estimate
  \begin{align*}
    \lefteqn{\biggabs{\frac{d}{dt}  \eta_{t\frac{\abs{x-y}}{2}}\big([x,y]_t-z\big)}} \quad
    &
    \\
    &\lesssim \indicator_{\set{\abs{[x,y]_t-z} \leq t \frac{\abs{x-y}}{2}}} \bigg((1-\alpha) t^{-1} \bigabs{[x,y]_t-z}^{-n+1-\alpha} \big( t \abs{x-y}\big)^{\alpha-1}
      \\
    &\qquad \hphantom{\indicator_{\set{\abs{[x,y]_t-z} \leq t \frac{\abs{x-y}}{2}}}} \!\!+(\abs{x\!-\!y}t)^{-n} (1\!-\!\alpha) \bigabs{[x,y]_t-z}^{-n-\alpha} (t \abs{x-y})^{n+\alpha} \frac{\abs{x-z}}{\abs{x-y} t^2} \bigg).
  \end{align*}
  Note that if $\abs{[x,y]_t-z} \leq t \frac{\abs{x-y}}{2}$, then
  \begin{align*}
    \abs{x-z}\geq \abs{x-[x,y]_t}-\abs{[x,y]_t-z}=t\abs{x-y}-\abs{[x,y]_t-z}\geq \abs{[x,y]_t-z}.
  \end{align*}
  This and the previous estimate imply
  \begin{align*}
    \biggabs{\frac{d}{dt}  \eta_{t\frac{\abs{x-y}}{2}}\big([x,y]_t-z\big)}\lesssim \indicator_{\set{\abs{[x,y]_t-z} \leq t \frac{\abs{x-y}}{2}}} (1-\alpha)t^{\alpha -2}\frac{\abs{x-y}^{\alpha-1} \abs{x-z}}{\bigabs{[x,y]_t-z}^{n+\alpha}}.
  \end{align*}
  Overall, we obtain
  \begin{align*}
    \lefteqn{\abs{v(x) - (v*\eta_{\frac{\abs{x-y}}{2}})(y)}} \quad &
    \\
    &\lesssim (1-\alpha)\int_0^1\!\!\int \indicator_{\set{\abs{[x,y]_t-z} \leq t \frac{\abs{x-y}}{2}}} |v([x,y]_t)-v(z)|\frac{t^{\alpha -2}\abs{x-y}^{\alpha-1} \abs{x-z}}{\bigabs{[x,y]_t-z}^{n+\alpha}}    \, dz\,dt .
  \end{align*}
  Hence,
  \begin{align*}
    \mathrm{I}&\lesssim (1-\alpha)\fint_{B_1}\!\!\int\!\!\int_0^1 \indicator_{\set{\abs{[x,y]_t-z} \leq t \frac{\abs{x-y}}{2}}} |v([x,y]_t)-v(z)|t^{\alpha -2}\frac{\abs{x-y}^{\alpha-1} \abs{x-z}}{\bigabs{[x,y]_t-z}^{n+\alpha}}   \,dt\, dz\,dy
  \end{align*}
  We now substitute $\zeta= x+t(y-x)=[x,y]_t$. Then $dy=t^{-n}d\zeta$ and $t=\frac{\abs{\zeta-x}}{\abs{y-x}}\geq \frac{\abs{\zeta-x}}{2}$. Using additionally Fubini, we get
  \begin{align*}
    \mathrm{I}&\lesssim (1-\alpha)\int\!\!\fint_{B_2}\int_{\frac{\abs{\zeta-x}}{2}}^1 \indicator_{\set{\abs{\zeta-z} \leq \frac{\abs{\zeta -x}}{2}}}|v(\zeta)-v(z)| t^{-n -1}\frac{\abs{\zeta-x}^{\alpha-1} \abs{x-z}}{\bigabs{\zeta-z}^{n+\alpha}} \,dt \, d\zeta\,dz
    \\
              &\lesssim (1-\alpha)\int\!\!\fint_{B_2} \indicator_{\set{\abs{\zeta-z} \leq \frac{\abs{\zeta -x}}{2}}}|v(\zeta)-v(z)| \frac{\abs{\zeta-x}^{-n+\alpha-1} \abs{x-z}}{\bigabs{\zeta-z}^{n+\alpha}}  \, d\zeta\,dz .
  \end{align*}
  Finally notice that if $\abs{\zeta-z} \leq \tfrac 12  \abs{\zeta -x}$, then
  \begin{align*}
    \tfrac 12 \abs{x-\zeta}\leq \abs{x-z}\leq \tfrac 32 \abs{x-\zeta}.
  \end{align*}
  This yields
  \begin{align*}
    \mathrm{I}&\lesssim (1-\alpha)\int_{B_2}\!\fint_{B_2} \frac{|v(\zeta)-v(z)|}{\abs{\zeta-z}^{n+\alpha}(\abs{x-z}+\abs{x-\zeta})^{n-\alpha}}  \, d\zeta\,dz .
  \end{align*}
  This finishes our estimates for $\mathrm{I}$. Next we estimate $\mathrm{II}$. We have
  \begin{align*}
    \mathrm{II}&\lesssim \fint_{B_1}\!\!\int \abs{v(y)-v(z)} \eta_{ \frac{\abs{x-y}}{2}}(y-z)\,dz\,dy\\
               &\lesssim (1-\alpha)\fint_{B_2}\!\int _{B_2}\indicator_{\set{\abs{y-z}\leq \frac 12 \abs{x-y}}}\abs{v(y)-v(z)} \abs{y-z}^{-n+1-\alpha}\abs{x-y}^{\alpha -1}\,dz\,dy
               \\
               &\lesssim  (1-\alpha)\fint_{B_2}\!\int _{B_2}\frac{\abs{v(y)-v(z)}}{\abs{y-z}^{n}}\,dz\,dy
               \\
               &\lesssim(1-\alpha) \int_{B_2}\!\int_{B_2} \frac{\abs{v(y)-v(z)}}{\abs{y-z}^{n+\alpha}(\abs{x-y}+\abs{x-z})^{n-\alpha}}\,dz\,dy.
  \end{align*}
  Finally, we estimate $\mathrm{III}$. Here we have
  \begin{align*}
    \mathrm{III}&\lesssim \fint_{B_1}\!\!\int \abs{v(y)-v(z)} \eta_{1}(y-z)\,dz\,dy\\
                &\lesssim (1-\alpha)\fint_{B_2}\!\int_{B_2}\abs{v(y)-v(z)} \abs{y-z}^{-n+1-\alpha}\,dz\,dy
                \\
                &\lesssim(1-\alpha) \int_{B_2}\!\int_{B_2} \frac{\abs{v(y)-v(z)}}{\abs{y-z}^{n+\alpha}(\abs{x-y}+\abs{x-z})^{n-\alpha}}\,dz\,dy.
  \end{align*}
  This finishes the proof in the case $\psi = \omega$.

  \textbf{Step 2:} Let now $\psi$ be a general weight as in the statement of the lemma and $\omega$ as in Step 1. Then
  \begin{align*}
    \abs{v(x)-\mean{v}_\psi} \leq \abs{v(x)-\mean{v}_\omega}+\abs{\mean{v}_\omega-\mean{v}_\psi}.
  \end{align*}
  We already estimated the first summand in Step 1. For the second term we have again by Step 1 and the bound on $\psi$
  \begin{align*}
    \abs{\mean{v}_\omega-\mean{v}_\psi}
    &\leq \int_{B_2} \abs{v(y)-\mean{v}_{\omega}} \psi(y) \,dy
    \\
    &\lesssim (1-\alpha)\int_{B_2} \int_{B_2}\!\int_{B_2} \frac{\abs{v(z)-v(\zeta)}}{\abs{z-\zeta}^{n+\alpha}(\abs{y-z}+\abs{y-\zeta})^{n-\alpha}}\,dz\,d\zeta \, \psi(y)\,dy
    \\
    &=(1-\alpha)\fint_{B_2}\!\int_{B_2} \frac{\abs{v(z)-v(\zeta)}}{\abs{z-\zeta}^{n+\alpha}} \int_{B_2} \frac{1}{(\abs{y-z}+\abs{y-\zeta})^{n-\alpha}} \,dy\,dz\,d\zeta
    \\
    &\lesssim \frac{1-\alpha}{\alpha}\fint_{B_2}\!\int_{B_2} \frac{\abs{v(z)-v(\zeta)}}{\abs{z-\zeta}^{n+\alpha}} \,dz\,d\zeta
    \\
    &\lesssim \frac{1-\alpha}{\alpha}\int_{B_2}\!\int_{B_2} \frac{\abs{v(z)-v(\zeta)}}{\abs{z-\zeta}^{n+\alpha}(\abs{x-z}+\abs{x-\zeta})^{n-\alpha}} \,dz\,d\zeta.
  \end{align*}
  This is the desired estimate for $\alpha \geq \frac 12$, since in this case we have $\frac 1\alpha \eqsim 1$. If $\alpha \in (0,\frac 12)$, then the factor $\frac{1-\alpha}{\alpha}$ blows up. In this case we can proceed as follows:
  \begin{align*}
    \abs{\mean{v}_\omega-\mean{v}_\psi}
  &\lesssim \int\int \abs{v(z)-v(y)}\omega(z)\psi(y)\,dz\,dy
  \\
  &\lesssim \fint_{B_2}\fint_{B_2}\abs{v(z)-v(y)} \,dz\,dy
  \\
  &\lesssim (1-\alpha)\int_{B_2}\!\int_{B_2} \frac{\abs{v(z)-v(y)}}{\abs{z-y}^{n+\alpha}(\abs{x-z}+\abs{x-y})^{n-\alpha}} \,dz\,dy ,
  \end{align*}
  where we used that $(1-\alpha)\eqsim 1$, $\abs{\psi}\lesssim 1$ and $\abs{\omega}\leq 1$. This finishes the proof.
\end{proof}
\begin{remark}\label{rem:Rieszpotential}
  The estimate in Lemma~\ref{lem:Rieszpotential} can be replaced by the estimate
  \begin{align}
    \label{eq:Rieszpotentialsimple}
    \abs{v(x) - \mean{v}_\psi}
    &\lesssim \int_B\!\!\int_B \frac{|v(\zeta)-v(z)|}{(\abs{x-z}+\abs{x-\zeta})^{2n}}  \, d\zeta\,dz
  \end{align}
  using a much simplified proof. However, only the factor $(1-\alpha)$ in~\eqref{eq:Rieszpotential} allows to prove $s$-stable estimates.
\end{remark}


\printbibliography
\end{document}